\documentclass[a4paper,11pt,openright,twoside]{article}

\usepackage{setspace}
\usepackage{color}
\usepackage{latexsym}
\usepackage{amssymb,amsbsy,amsmath,amsfonts,amssymb,amscd}
\usepackage{epsfig, graphicx}
\usepackage{xcolor}
\usepackage[numbers]{natbib}
\usepackage{natbib}
\PassOptionsToPackage{hyphens}{url}
\usepackage[colorlinks=true,
            linkcolor=red,
            urlcolor=blue,
            citecolor=blue]{hyperref}
\usepackage{enumerate}
\usepackage{mathrsfs}
\usepackage{color}

\usepackage{float}
\usepackage{tikz}
\usepackage[all]{xy}
\usepackage{wrapfig}
\usepackage{enumitem}
\usepackage{multirow, array} 

\usepackage{appendix}
\usepackage{fancyhdr}

\usepackage{caption,subcaption}

\newcommand{\commentout}[1]{}
\newcommand{\R}{\mathbb{R}}

\newcommand {\e}  {\mathrm{e}}

\newcommand {\Chi} {{\bf \raise 2pt \hbox{$\chi$}} }
\newcommand{\norme}[1]{\left\lVert#1\right\rVert}

\newcommand{\diff}{\mathrm{d}}

\newtheorem{theorem}{Theorem}[section]
\newtheorem{lemma}[theorem]{Lemma}
\newtheorem{definition}[theorem]{Definition}

\newtheorem{proposition}[theorem]{Proposition}
\newtheorem{corollary}[theorem]{Corollary}

\numberwithin{equation}{section}

\newcommand{\qed}{{ \hfill
                      {\unskip\kern 6pt\penalty 500 \raise -2pt\hbox{\vrule\vbox to 6pt{\hrule width 6pt
                      \vfill\hrule}\vrule} \par}   }}

\title{\Large \bf Adaptation to DNA damage, an asymptotic approach for a cooperative non-local system}
\author{
	Alexis Leculier\thanks{Laboratoire Jacques-Louis Lions (LJLL), Sorbonne Universit\'e, 75205 Paris Cedex 06, France.\texttt{leculier@ljll.math.upmc.fr}} \and	
Pierre Roux\thanks{Mathematical Institute, University of Oxford, OX2 6GG Oxford, United Kingdom.
	\texttt{pierre.roux@maths.ox.ac.uk}} 
\thanks{Previously: Laboratoire Jacques-Louis Lions (LJLL), Sorbonne Universit\'e, 75205 Paris Cedex 06, France.}
}

\date{\today}

\begin{document}
\maketitle           
\begin{abstract}
Following previous works about integro-differential equations of parabolic type modelling the Darwinian evolution of a population, we study a two-population system in the cooperative case. First, we provide a theoretical study of the limit of rare mutations and we prove that the limit is described by a constrained Hamilton-Jacobi equation. This equation is given by an eigenvalue of a matrix which accounts for the diffusion parameters and the coefficients of the system. Then, we focus on a particular application: the understanding of a phenomenon called Adaptation to DNA damage. In this framework, we provide several numerical simulations to illustrate our theoretical results and investigate mathematical and biological questions.
\end{abstract}


\noindent{\bf Key-words}: : Adaptive evolution, Cooperative system, Lotka-Volterra equation, Hamilton-Jacobi equation, Viscosity solutions .
\newline
\noindent
{\bf AMS Class. No}:  35K60, 82C31, 92B20, 35Q84




%

\section{Introduction}


A common way to investigate evolutionary dynamics \cite{D,BJvF} is to model populations structured by a phenotipical trait with non-local partial differential equations \cite{BMP,BP,PB}. This methodology has the advantage of studying not only the final situation but also the fitness landscape and the possible evolutionary paths in a given setting  \cite{PKWT}. In those kind of models, the organisms are described by a trait $x\in\R^n$ and their density $n(x,t)$ expands or decays in function of both $x$ and the competition with other individuals. A simple  possibility to represent mutations along the trait $x$ is to use a Laplacian:
\begin{equation}\label{eqn:simple}
    \varepsilon\dfrac{\partial n^\varepsilon}{\partial t}(x,t) = \varepsilon^2\Delta n^\varepsilon(x,t) + n^\varepsilon(x,t)R(x,N(t)), \qquad N_\varepsilon(t) = \int_{\R^n} n^\varepsilon(x,t) dx.
\end{equation}
This type of model can be derived from individual based stochastic models in the large population limit \cite{CFM1,CFM2}. The parameter $\varepsilon$ provides a way to study the asymptotic limit of the model in the regime of small mutations and long time \cite{CC}. This procedure relies upon an Hamilton-Jacobi approach and was investigated for system \eqref{eqn:simple} in, \textit{e.g.}, \cite{BP,PB,BMP}. Indeed, this equation relies on the change of variable
\[ (x,  t) = \left( \frac{x}{\varepsilon} ,\frac{t}{ \varepsilon }\right). \]
This change of variables allows us to catch the effective behaviour of the solutions in large timescales. In a suitable setting, when $\varepsilon\to 0$, the solutions $n^\varepsilon$ of  \eqref{eqn:simple} concentrate into a sum of Dirac masses moving in time, and in the limit the location of emergent traits is driven by an Hamilton-Jacobi equation of the form
\begin{equation}\label{eqn:simple_HJ}
    \dfrac{\partial u}{\partial t} =  |\nabla u |^2 + R(x,N(t)), \qquad \max_{x\in\R^n} u(x,t) = 0.
\end{equation}

This type of non-local model was intensively studied and applied to many different biological contexts, for example adaptation of cancer to treatment \cite{CLLLAEC, LCC, LLHCP, PCLT}, epigenetics changes \cite{LCDH}, non-inherited antibiotic resistance \cite{BvdHGKD} or more generally long-time evolutionary dynamics \cite{FM, JR, M2}. Finally, we underline that a more realistic approach is to use an integral term and a mutation kernel (see for instance \cite{BP, B} and the references therein) since, in our case, it is tantamount to saying that the mutations are independent of birth.

\subsection{A model for two cooperative populations structured by a phenotypical trait}

We propose to study through this Hamilton-Jacobi procedure a system of non-local PDEs modelling two cooperative populations structured by a same phenotypical trait $x\in\R_+$ and described by their densities $n_1^\varepsilon(x,t)$ and $n_2^\varepsilon(x,t)$. This model is motivated by a particular application in genetics: the understanding of the so-called "adaptation to DNA damage" phenomenon"\footnote{The naming "adaptation to DNA damage" can be a bit misleading because it describes a metabolic response of the cells and not a genetic adaptation, so speaking about both adaptation to DNA damage and genetic adaptation from the perspective of evolutionary dynamics can be confusing sometimes. Nonetheless, this expression was used for so long among specialists and has gained so much momentum that it is now impossible to change.}. When the DNA of an eukaryotic cell is damaged, the cell cycle is stopped by a checkpoint and repair pathways are activated. If repair fails, the cells may escape the DNA damage checkpoint and reenter the cell cycle despite the damage being still present: the cell "adapted" to its DNA damage \cite{TGH,LMHUKH}. By dividing a cell population into two categories (normal cells and adapted cells) we can use a two populations system to study the characteristic time and variance of this phenomenon.

To the best of our knowledge, there is little research about the asymptotic behaviour of several species non-local PDEs in evolutionary dynamics. Existing works in this direction focus, for instance, on the influence of a spacial domain \cite{BM}, on organisms which specialise in order to consume particular resources \cite{DJMP}, on a model for juvenile-adult population undergoing small mutations \cite{CCP}, or on elliptic systems \cite{M2, LM, PS} for two species or on influence of a spacial domain \cite{BM}. 

The model we focus on writes
\begin{equation}
    \label{HJ:eq:main}
    \left\lbrace
    \begin{aligned}
    &\varepsilon \partial_t n_1^\varepsilon- \varepsilon^2 d_1\partial_{xx} n_1^\varepsilon= n_1^\varepsilon(r_1(x)  - N_\varepsilon(t)) + \delta_1(x) n_2^\varepsilon \quad \text{ for } (x, t) \in \mathbb{R}^+ \times \mathbb{R}^+,\\
    &\varepsilon \partial_t n_2^\varepsilon  - \varepsilon^2 \partial_{xx} d_2 n_2^\varepsilon = n_2^\varepsilon(r_2(x)  - N_\varepsilon(t) )+ \delta_2(x) n_1^\varepsilon\quad \text{ for } (x, t) \in \mathbb{R}^+ \times \mathbb{R}^+,\\
    &n_1(x, t=0) = n_{1,0}^\varepsilon(x), \qquad n_2^\varepsilon(x, t=0) = n_{2, 0}^0(x), \\
    &\partial_x n_1(x=0, t) = 0, \qquad \partial_x n_2^\varepsilon(x=0, t) = 0, \\
    &N_\varepsilon(t) = \int_{0}^{+\infty} \big( n_1^\varepsilon  (x,t) + n_2^\varepsilon(x,t) \big)dx,
    \end{aligned}\right.
\end{equation}
where $r_1(x)\geqslant0$ and $r_2(x)\geqslant0$ represent the intrinsic fitness of organisms with trait $x$ in the two populations. The terms $\delta_1(x)\geqslant 0$ and $\delta_2(x)\geqslant0$ are cooperative terms (or, in our application in Section \ref{sec5:Motiv}, conversion terms from one cell type to the other) between the two populations. The total number of cells $N_\varepsilon(t)$ represents the competition for resources.

The system can be summarised in the following compact form 
\begin{equation}
    \label{HJ:eq:main2}
    \varepsilon \partial_t \textbf{n}^\varepsilon - \varepsilon^2 \textbf{D} \partial_{xx} \textbf{n}^\varepsilon = \textbf{R}(x, N_\varepsilon) \textbf{n}^\varepsilon,
\end{equation}
with boundary conditions. Here $\textbf{n}^\varepsilon$ stands for the vector $(n_1^\varepsilon, n_2^\varepsilon)^T$ and $\textbf{D}, \ \textbf{R}$ for the following operators:
\begin{equation}
    \textbf{D} = \begin{pmatrix}  d_1 & 0 \\ 0 & d_2 \end{pmatrix}\quad \text{ and } \quad  
    \textbf{R}(x, N) = \begin{pmatrix} r_1(x) - N & \delta_1(x) \\ \delta_2(x) & r_2(x) - N \end{pmatrix}. 
\end{equation}
First, we assume that
\begin{equation}\label{HJ:Hyp:di}\tag{H1}
    d_1 , \ d_2 \geq 0 \quad \text{ and } \quad (d_1 , d_2 ) \neq (0, 0). 
\end{equation} 
Note that \eqref{HJ:Hyp:di} allows one of the two coefficients $d_1,d_2$ being equal to $0$, but not both at the same time. We will also assume that there exists $C_R, C_\delta>0$ such that
\begin{equation}
    \label{HJ:Hyp:Rdelta}\tag{H2}
    \begin{aligned}
    &\delta_i, r_i \in W^{2, \infty} \quad \text{ with } \quad  \| r_i \|_{W^{2, \infty}} \leq C_R \quad \text{ and } \| \delta_i \|_{W^{2, \infty}} \leq C_\delta,   \\
    &\delta_i > 0,  \quad \text{ and } \quad e^{C_\delta x}\delta_i(x) \underset{x \to +\infty}{\longrightarrow} +\infty.
    \end{aligned}
\end{equation}
An other hypothesis is
\begin{equation}
\label{HJ:Hyp:R}\tag{H3}
\begin{aligned}
\exists c_N, C_N>0: \qquad\forall x \in \mathbb{R}^+, \  &\min(r_1(x) + \delta_2(x)  -c_N, \ r_2(x) +\delta_1(x)  -c_N) \geq  0, \\
\forall x \in \mathbb{R}^+, \   &\max (r_1(x) + \delta_2(x)  -C_N, \ r_2(x) + \delta_1(x)  -C_N) \leq  0.  
\end{aligned}
\end{equation}
Finally, we assume that both initial conditions satisfy:
\begin{equation}
    \label{HJ:Hyp:Init_cond}\tag{H4}
     \begin{aligned}
    &ce^{\frac{{-ax^2-c}}{\varepsilon}} \leq n_{i, 0}^\varepsilon(x) \leq Ce^{\frac{{-Ax + C}}{\varepsilon}} \qquad \text{ with } \quad a, A, c, C>0,\\
    &c_N \leq N_\varepsilon (t=0) \leq C_N \quad \text{ and } \quad  n_{i,0}^\varepsilon \text{ are uniformly Lipshitz}.
    \end{aligned}
\end{equation}

\begin{theorem}\label{HJ:thm:existence}
Under the assumptions \eqref{HJ:Hyp:di}, \eqref{HJ:Hyp:Rdelta}, \eqref{HJ:Hyp:R} and \eqref{HJ:Hyp:Init_cond}, there exists a solution $\textbf{n}^\varepsilon$ to \eqref{HJ:eq:main}. Moreover, we have 
\[ c_N \leq N_\varepsilon(t) \leq C_N. \]
\end{theorem}

The proof is an adaptation of the one presented in Appendix A of \cite{BMP}. We provide it in the Appendix for the sake of completeness.

\subsection{The main result}

We adopt the classical approach for Hamilton-Jacobi equations: we perform the so-called Hopf-Cole transformation by defining
\begin{equation}
    \label{HJ:HC}
    u_i^\varepsilon = \varepsilon \ln (n_i^\varepsilon), 
\end{equation}
in a such a way that if $n_i^\varepsilon$ converges to a Dirac mass at some point $(x_0,t_0)$, then it is sufficient to prove that $u_i(x_0,t_0)=0$ whereas $u_i(x,t)<0$ for $(x,t) \in B_r(x_0, t_0) \backslash \left\lbrace (x_0, t_0) \right\rbrace$. Therefore, we rewrite \eqref{HJ:eq:main} in the following form
\begin{equation}
    \label{HJ:eq:main:HC}
    \left\lbrace
    \begin{aligned}
    & \partial_t u_1^\varepsilon- \varepsilon d_1\partial_{xx} u_1^\varepsilon -  d_1[\partial_x u_1^\varepsilon]^2= (r_1(x)  - N_\varepsilon(t)) + \delta_1(x) e^{\frac{u_2^\varepsilon - u_1^\varepsilon}{\varepsilon}} \quad \text{ for } (x, t) \in \mathbb{R}^+ \times \mathbb{R}^+,\\
    &\partial_t u_2^\varepsilon- \varepsilon d_2\partial_{xx} u_2^\varepsilon - d_2 [\partial_x u_2^\varepsilon]^2= (r_2(x)  - N_\varepsilon(t)) + \delta_2(x) e^{\frac{u_1^\varepsilon - u_2^\varepsilon}{\varepsilon}}\quad \text{ for } (x, t) \in \mathbb{R}^+ \times \mathbb{R}^+,\\
    &u_1^\varepsilon(x, t=0) = u_{1,0}(x), \qquad u_2^\varepsilon(x, t=0) = u_{2,0}(x), \\
    &\partial_x u_1(x=0, t) = 0, \qquad \partial_x u_2^\varepsilon(x=0, t) = 0, \\
    &N_\varepsilon(t) = \int_{0}^{+\infty} \left( e^{\frac{u_1^\varepsilon(x,t)} {\varepsilon}} + e^{\frac{u_2^\varepsilon (x,t)}{\varepsilon}}\right) dx.
    \end{aligned}\right.
\end{equation}

Finally, following \cite{BES}, we introduce the effective Hamiltonian as known as one of the eigenvalue of $\rho^2\textbf{D} + \textbf{R}$ (associated to a constant sign eigen-vector):
\begin{equation}
    \label{HJ:def:H}
    \mathcal{H}_D(\rho, N) =\frac{d_1 + d_2}{2}\rho^2 +  \frac{ r_1 + r_2 + \sqrt{[(d_1 - d_2)\rho^2 +  (r_1 - r_2)]^2 + 4 \delta_1 \delta_2 }}{2 } - N(t). 
\end{equation}
We introduce the \textit{Hamiltonian fitness}
\begin{equation}\label{HJ:def:rh}
r_H^D (x, \rho) = \frac{ r_1 + r_2 + \sqrt{[(d_1 - d_2)\rho^2 +  (r_1 - r_2)]^2 + 4 \delta_1 \delta_2 }}{2 }
\end{equation}
such that 
\[\mathcal{H}_D(\rho, N) =\frac{d_1 + d_2}{2}\rho^2 + r_H^D (x,\rho) - N. \]
We will denote $\psi^\rho$ the corresponding principal eigen-vector:
\begin{equation}
    \label{HJ:def:psi}
    \psi^{\rho}(x) = \begin{pmatrix} 1 \\ \frac{ (d_1 - d_2) \rho^2 +  (r_1(x) - r_2(x))  +  \sqrt{((d_1 - d_2) \rho^2 + r_1(x) - r_2(x))^2 + 4 \delta_1(x) \delta_2(x) }}{2\delta_2(x)}\end{pmatrix}.
\end{equation}
All the components of $\psi^\rho$ can be chosen strictly positive. The other eigenvector, associated to the eigenvalue $\frac{ (d_1 + d_2)\rho^2 + r_1 + r_2 - \sqrt{ (d_1 - d_2)\rho^2 + [(r_1 - r_2) ]^2 + 4 \delta_1 \delta_2 }}{2 } - N(t)$, has a positive and a negative component.

\begin{theorem}\label{HJ:thm:main_thm}
Under the hypotheses \eqref{HJ:Hyp:di}, \eqref{HJ:Hyp:Rdelta}, \eqref{HJ:Hyp:R} and \eqref{HJ:Hyp:Init_cond}, there hold 
\begin{enumerate}
    \item The sequence $(N_\varepsilon)_{\varepsilon > 0}$ converges to a non-decreasing function $N \in L^\infty(\,]0, +\infty[\,)$ as $\varepsilon \to 0$ with 
    \[ c_N \leq N(t) \leq C_N.\]
    \item The sequence $(u_i^\varepsilon)_{\varepsilon>0, \ i \in \left\lbrace1,2 \right\rbrace}$ converges locally uniformly to a same continuous function $u$, with $u$ a viscosity solution of 
    \begin{equation}
        \label{HJ:eq:main_eq}
        \left\lbrace
        \begin{aligned}
         &\partial_t u = \mathcal{H}_D(\partial_x u , N) && \text{ for } (x,t) \in \mathbb{R}_+ \times\, ]0, +\infty[, \\
         &- \partial_x u (x=0, t)  = 0 && \text{ for } t>0,\\
         &\underset{ x \in \mathbb{R}_+}{\max}  \   u(x,t) = 0 ,\\
         &u(x,t=0) = \underset{\varepsilon\to0}{\lim} \  u_i^\varepsilon (x, t=\varepsilon ).
        \end{aligned}
        \right.
    \end{equation}
    \item The sequence $(n_i^\varepsilon)_{\varepsilon >0, \ i \in \left\lbrace 1,2 \right\rbrace}$ converges in the sense of measures to $n_i$. Moreover, we have 
    \[ \mathrm{supp} \  \ n_i( \cdot, t) \subset \left\lbrace u( \cdot , t) = 0 \right\rbrace. \]
\end{enumerate}
\end{theorem}

\subsection{ Outline of the paper  }

In Section \ref{sec2:Approach}, we detail the general approach and state the main technical results that lead to the proof of Theorem\ref{HJ:thm:main_thm}. Section \ref{Sec3:interm} is devoted to the proofs of these technical results. In section \ref{Sec4:HJ}, we prove Theorem\ref{HJ:thm:main_thm}. Next, in section \ref{sec5:Motiv}, we detail the biological context that motivates our theoretical study. Finally, in section \ref{sec6:Num}, we illustrate our theoretical study by some numerical simulations in the framework given by our biological motivations. We also investigate numerically some open questions. 

\vspace{0.5cm}

\noindent \textbf{Notations :} All along the paper, we adopt the following conventions:
\begin{itemize}
\item the letters $i,j$ refer, when there is no confusion possible, to an index in $\left\lbrace 1,2 \right\rbrace$,
\item if $i$ and $j$ are used in a same equation then $i \neq j$, 
\item the bold mathematical characters are strictly reserved for vectors of $\mathbb{R}^2$ or matrix of $\mathcal{M}_2(\mathbb{R})$, 
\item the constants $c, C$ are taken positive and may change from line to line when there is no confusion possible (the capital letter is preferentially used for large constants and the small letter for small constants). 
\end{itemize} 

\section{The Hamilton Jacobi approach}\label{sec2:Approach}

We develop in this part of the work a general approach for non-local cooperative systems. For technical reasons, we focus on a model with only two species and a uni-dimensional space. This part is largely inspired by \cite{BES} and \cite{BMP}, but since we study a coupled system, we cannot use the same arguments straight away. Unlike in the articles \cite{BP} and \cite{BMP} for the single species problem, it is not possible to obtain directly a uniform BV estimate for the total mass $N_\varepsilon(t)$.  There are additional mixing terms and, \textit{a priori}, nothing prevents them to blow-up when $\varepsilon$ goes to $0$. Moreover, one can not apply directly the method of \cite{BES} because the non-local total mass does not prevent the logarithm of the solution to be positive. We will circumvent these issues by employing a combination of the two former approaches.

\subsection{The approach}

Before dealing with the mathematical details, we propose an overview of the classical methods to treat this kind of problem as well as a presentation of heuristic arguments. 

A local version of \eqref{HJ:eq:main2} was studied in \cite{BES} (i.e. with $N_\varepsilon$ replaced by $(n_i^\varepsilon)_{i \in \left\lbrace1,2 \right\rbrace}$). Moreover, the authors focus on general systems with more that two equations. We do not obtain the same level of generality than \cite{BES}. As we will see later, the hypothesis of having only two equations (rather than several) is a key hypothesis in our work. 
From a technical point of view, Barles, Evans and Souganidis do not prove any regularity results on $u_i^\varepsilon$ but they study the system through the semi-relaxed limit method by defining 
\[ u_* (x,t) = \underset{i \in \left\lbrace 1,2 \right\rbrace}{\min}(\underset{(y,s) \to (x,t)}{\underset{\varepsilon \to 0}{\liminf} } \  u_i^\varepsilon(y,s)) \qquad \text{ and } \qquad u^*(x,t) = \underset{i \in \left\lbrace 1,2 \right\rbrace}{\max}(\underset{(y,s) \to (x,t)}{\underset{\varepsilon \to 0}{\limsup} } \  u_i^\varepsilon(y,s)). \]
We did not succeed in adapting this idea without proving any regularity results on $u_i^\varepsilon$. Indeed, with the semi-relaxed limit approach, one key point is to prove that $u^* \leq 0$. In \cite{BES}, this claim is true; otherwise, it would be in contradiction with some natural bounds on $n_\varepsilon$ (obtained with the maximum principle). However, in our setting without any regularity result in space on $u_i^\varepsilon$, even if we have natural bounds on the total mass $N_\varepsilon$, nothing prevents the solution $u^*$ to be positive at a singular point. Indeed, contrary to the problem studied in \cite{BES}, $u_i^\varepsilon$ may be positive on a sequence of intervals $I_\varepsilon$ with $\lambda(I_\varepsilon ) \to 0$ (where $\lambda$ stands for the Lebesgue measure). 

Therefore, we state regularity results in space on $u_i^\varepsilon$. Our result generalizes the case of the single population equation \eqref{eqn:simple} (\textit{i.e.} $\delta_i = 0$ and $n_2 = 0$). In the first works treating this equation \cite{BP,PB,BMP}, the main result on the convergence of $u_\varepsilon$ was obtained by proving some BV-estimates on $N_\varepsilon$ and some bounds on $|\partial_x u_1^\varepsilon|$ by using the Bernstein method. Then obtaining the Lipschitz regularity of $u_1^\varepsilon$ with respect to time leads to the convergence by using the Arzela-Ascoli Theorem. Before, dealing with the Hamilton-Jacobi equation \eqref{HJ:eq:main_eq}, we prove the convergence of $N_\varepsilon$ toward subsequence. We adapt the proof of \cite{PB} (Theorem 3.1) and \cite{BMP} (Theorem 2.4). The proof of the Theorem of \cite{PB} involves the positiveness of $r^2$ (equation (3.5) of \cite{PB}). In our work, it is not clear in general that
\[0 \leq  \begin{pmatrix} 1 & 1 \end{pmatrix}\textbf{R}^2  \begin{pmatrix} n_1 \\ n_2 \end{pmatrix},    \]
the right-hand side being what we would obtain in place of $r^2$.   

To tackle this issue, we propose a precise estimate of $\frac{n_1^\varepsilon}{n_2^\varepsilon}$. Indeed, this estimate ensures that the exponential term is bounded and then one can apply the classical Bernstein method to obtain regularity in space. From this space regularity, we will deduce that $u_i^\varepsilon$ is Lipschitz with respect to time. Finally from this last result, we deduce that the family $N_\varepsilon$ converges. It will allow us to conclude. \\
We underline that the estimate of $\frac{n_1^\varepsilon}{n_2^\varepsilon}$ plays a similar role than the Harnack estimates obtained in \cite{M,LM} in elliptic settings. 

We formally write a Taylor expansion of $u_i^\varepsilon$:
\[ u_i^\varepsilon = u_i + \varepsilon v_i + o(\varepsilon).\]
We first expect that $u_1 = u_2 = u$ since we do not expect a blow up of the exponential term. Next, by subtracting the two equations and using the fact that $u_1 = u_2$, we obtain
\[ \begin{aligned}\big[(d_1 - d_2)(\partial_x u)^2 + (r_1 - r_2 ) \big]\frac{n_1}{n_2} + \delta_1  - \delta_2 \left(\frac{n_1}{n_2}\right)^2 = \varepsilon \big(\partial_t (v_1 - v_2)  + &  d_1[\partial_x v_1]^2 - d_2[\partial_x v_2]^2  \\
& - 2\partial_x u \partial_x(d_1v_1 - d_2v_2) +o(\varepsilon) \big). \end{aligned}\]
Taking formally, the limit $\varepsilon \to 0$, we expect
\[\frac{n_1 }{n_2} \sim \frac{[(d_1 - d_2)(\partial_x u)^2 + (r_1 - r_2 ) ] + \sqrt{[(d_1 - d_2)(\partial_x u)^2 + (r_1 - r_2 ) ]^2 + 4\delta_1\delta_2} }{2 \delta_2}.    \] 
The above expression involves $\partial_x u$ which is not clearly defined yet. Notice here that in the special case $d_1 = d_2$ the formula is simpler since the right-hand part is only defined thanks to the cfunctions $r_1,r_2,\delta_1,\delta_2$ and we expect
\[\frac{n_1 }{n_2} \sim \frac{ (r_1 - r_2 )  + \sqrt{ (r_1 - r_2 )^2 + 4\delta_1\delta_2} }{2 \delta_2}.\] 

\begin{definition}
Let $q_i$ be the unique positive root of 
\[P_{d_i, d_j}(X) = ([d_i -d_j](\partial_x u_j^\varepsilon)^2 + r_i -r_j)X+ \delta_i - \delta_j X^2,\]
\begin{equation}
  \text{ i.e. }   \label{HJ:app:eq:q}
    q_i = \frac{ ([d_i - d_j ] (\partial_x u_j^\varepsilon)^2 + r_i - r_j) + \sqrt{([d_i - d_j ] (\partial_x u_j^\varepsilon)^2 + r_i - r_j)^2 + 4 \delta_i \delta_j}}{2 \delta_j }.
\end{equation}
\end{definition}
The fact that $\mathrm{deg}(P_q) = 2$ is important because it allows us to make a reasoning on the sign of $P_q$. Next, with this definition, we state the main technical statements that are necessary to prove Theorem\ref{HJ:thm:main_thm}. 

\begin{theorem}\label{HJ:thm:interm_u}
Under the hypotheses \eqref{HJ:Hyp:di}, \eqref{HJ:Hyp:Rdelta}, \eqref{HJ:Hyp:R} and \eqref{HJ:Hyp:Init_cond}, the following assertions hold true.
\begin{enumerate}
    \item \textup{\texttt{Bounds. } } There exists $a',A',b,B$ such that 
    \begin{equation}
    \label{HJ:eq:ubd}
    -bt - a'x^2 -c \leq  u_i^\varepsilon (x,t)\leq  Bt - A'x + C. 
    \end{equation}
    \item \textup{\texttt{Space regularity. } } For any times $0<t_1<T$ and $R>0$ there exists a constant $C_{t_1, T, R}>0$ such that 
    \begin{equation}
    \underset{(x,t) \in [0, R] \times [t_1, T]}{\max}|\partial_x u_i^\varepsilon(x,t)| \leq C_{t_1, T, R}.
    \end{equation}
    \item \textup{\texttt{Ratio $\frac{n_1}{n_2}$. } } For any positive time, we have 
    \begin{equation}\label{hj:eq:ratio_q}
\begin{aligned}
&\varepsilon \left[ \ln (q_1(x, t)) - \frac{\varepsilon^4}{t} \right] \leq u_1^\varepsilon(x,t) - u_2^\varepsilon (x,t) \leq \varepsilon \left[ \ln (q_1(x, t)) + \frac{\varepsilon^4 }{t} \right]\\ 
\hspace{-1cm} \text{ and } \qquad& \varepsilon \left[ \ln (q_2(x,t)) - \frac{\varepsilon^4}{t} \right] \leq u_2^\varepsilon(x,t) - u_1^\varepsilon (x,t) \leq \varepsilon \left[ \ln (q_2(x,t)) + \frac{\varepsilon^4 }{t} \right].
\end{aligned}
\end{equation}
    \item \textup{\texttt{Time regularity. } } The family $(u_i^\varepsilon)_{\varepsilon>0 , i\in \left\lbrace 1, 2 \right\rbrace}$ is locally uniformly continuous with respect to time. 
\end{enumerate}
\end{theorem}

Remark that the third item (the ratio estimates) comes after the space regularity result since when $d_1<d_2$, if $\partial_x u_i^\varepsilon$ is not locally bounded with respect to $\varepsilon$, one can not conclude the proof of \eqref{hj:eq:ratio_q}. However, to prove the space regularity result, one needs an estimate similar to \eqref{hj:eq:ratio_q}. We prove a weaker version of \eqref{hj:eq:ratio_q} as an intermediate result but we state only the stronger result in the theorem above. We also highlight that the terms $\varepsilon^4$ has an exponent 4 that will be used in the proof of point 1. of Theorem\ref{HJ:thm:main_thm}.

\subsection{The special case $d_1 = d_2 = 1$}\label{HJ:sec:LongTermBehavior}

In this special setting, note that $q_i$ does not involve $\partial_x u_j^\varepsilon$ anymore. Therefore, the point 3. of Theorem\ref{HJ:thm:interm_u} can be obtained directly by observing that 
\[ - C(x+1) \leq \ln(q(x)) \leq C (x+1)\]
(for some large constant $C>0$). We refer to the forthcoming proof of Lemma \ref{HJ:lemma:weakAsympt} for more details. It follows that the point 2. of Theorem\ref{HJ:thm:interm_u} can be obtained from the point 3. 

\vspace{0.2cm}

Last, we can also derive formally a simpler equivalent equation for system \eqref{eqn:reduced_model} in the long time limit. We can assume $n_{1,\infty}(x)\simeq q(x)n_{2,\infty}(x)$ when $t\to +\infty$ and thus the quantity 
\[w(x) = n_{1,\infty}(x)+ n_{2,\infty}(x) = n_{2,\infty}(x)(1+q(x))\]
should satisfy the equation
\begin{equation}\label{eqn:scalar}
    - \varepsilon^2 \dfrac{\partial^2 w}{\partial x^2}(x) = w(x) \big(r_\infty(x) - N(t) \big), 
\end{equation}
with $N(t) =\int_0^{+\infty}w(x,t)dx$ and where the global fitness function $r_\infty$ of the system writes
\[ r_\infty(x) = \dfrac{q(x)}{1+q(x)} \left( r_1(x) + \delta_2(x)\right) + \dfrac{1}{1+q(x)}(r_2(x)+\delta_1(x)). \]
Equation \eqref{eqn:scalar} is well understood. It is proved in \cite{LP,AV} that for each $\varepsilon$ there exists a unique solution which is the ground state of the Schroedinger operator 
\[  \hat H:= - \varepsilon^2 \Delta - r_\infty.    \]
First, we remark that 
\[q =  \frac{[r_1 - r_2] + \sqrt{[r_1 - r_2]^2 + 4 \delta_1 \delta_2 }}{2 \delta_2} \quad \text{ and } \quad q^{-1} = \frac{[r_2 - r_1] + \sqrt{[r_2 - r_1]^2 + 4 \delta_1 \delta_2 }}{2 \delta_1}.\]
Recalling the definition \eqref{HJ:def:rh}, we notice that
\[  r_H^{I_2} = \delta_2 q + r_2 = \delta_1 q^{-1} + r_1.\] 
We conclude
 \begin{equation}  \label{HJ:eq:rHrinfty}
r_\infty = \dfrac{1}{1+q}\big(  q[r_1 + q^{-1}\delta_1] + [r_2 + q \delta_2] \big) = r_H^{I_2}(x). 
\end{equation}

Hence, the Hamiltonian fitness referred to above describes the behaviour of the system in both the limits $\varepsilon\to0$ and $t\to+\infty$. This function is formally the equivalent fitness of the overall system formed by the two cooperating populations. They adjust their fitness parameter $x$ in function of the maximum points of $r_H(x)$.

\section{The intermediate technical results}\label{Sec3:interm}

Here, we prove all the statements of Theorem\ref{HJ:thm:interm_u} and some intermediate results that are not stated above. 

\subsection{Bounds on $u_i^\varepsilon$}

First, we focus on the bounds for $u_i^\varepsilon$. The method is quite standard but some new difficulties arise from the interplay between the two populations.

\begin{proof}[Proof of 1. of\ref{HJ:thm:interm_u}.]
We split the proof into two parts : the upper bound and then the lower one. 

\textbf{The upper bound. } First, we define $\psi  = - A'x+ Bt + C$ with $A',B>0$ and $A' < A$ that will be fixed later on. We also introduce $w^\varepsilon = \max(u_1^\varepsilon, u_2^\varepsilon)$ and $i \in \left\lbrace 1,2 \right\rbrace$ the corresponding integer. 
 From assumption \eqref{HJ:Hyp:Init_cond}, it is clear that $w^\varepsilon (t= 0) \leq \psi(x, t= 0)$. Next, we consider 
\[ T := \inf \left\lbrace t>0: \quad \exists x > 0, \ w^\varepsilon (x,t) > \psi(x,t) \right\rbrace.\]
We prove by contradiction that $T = +\infty$. Assume $T < +\infty$. We distinguish two cases :
\begin{itemize}
    \item \textbf{Case 1 : }\textit{There exists $x_0>0$ such that $w^\varepsilon(x_0, T) = \psi(x_0, T)$. } It follows by definition of $T$ 
    \[ \partial_t (w^\varepsilon - \psi)(x_0, T) \geq 0, \quad - d_i\partial_{xx} (w^\varepsilon - \psi)(x_0, T) \geq 0, \quad \text{ and } \quad  d_i  \partial_x w^\varepsilon(x_0, T) =   \partial_x \psi(x_0, T). \]
The definition of $w^\varepsilon$ yields that the exponential part is bounded by $1$. From this bound and \eqref{HJ:Hyp:Rdelta}, it follows 
\[B -  A'^2 \leq \left( \partial_t \psi - \varepsilon d_i  \partial_{xx}  \psi -  ( d_i \partial_x \psi)^2 \right) (x_0,  T) \leq  \left( \partial_t w^\varepsilon - \varepsilon  d_i \partial_{xx}  w^\varepsilon -  ( d_i \partial_x w^\varepsilon)^2 \right) (x_0, T) \leq    C_R + C_\delta \]
which is impossible for $B >   A'^2 + C_R  + C_\delta$. (We remark that $x_0>0$ according to the Neumann boundary conditions imposed on $w^\varepsilon$. Moreover, the first inequality above is a strict inequality whenever $d_i = 0$.) 

    \item \textbf{Case 2 : }\textit{There holds $\inf (w^\varepsilon - \psi)(\cdot, T) = 0$ with $w^\varepsilon(\cdot, T) < \psi(\cdot, T)$. } In this case, we introduce $\psi_\gamma := \psi(x,t) - \gamma e^{-\frac{\varepsilon}{t}}$ with $\gamma \in (0, 1]$ and 
    \[ T_\gamma := \inf \left\lbrace t>0: \quad \exists x > 0, \  w^\varepsilon (x,t) > \psi_\gamma(x,t) \right\rbrace. \]
   Since $\psi_\gamma (t= 0^+) = \psi(t= 0)$ and $\psi_\gamma (T) = \psi(T)$, we have $0<T_\gamma <T$. Remark also that $T_{1} \leq  T_\gamma$ for $\gamma < 1$. Moreover, since $\partial_t \psi_\gamma (\cdot, t) =  B -\frac{\gamma e^{-\frac{\varepsilon}{t}}}{t^2} < \partial_t \psi (\cdot, t)$, we conclude as in case 1 that $w^\varepsilon (\cdot, T_\gamma)< \psi_\gamma(\cdot, T_\gamma)$. We claim that $T_\gamma < T$. Indeed, since there exists, by definition of $T$, a sequence $x_n>0$ such that $(w^\varepsilon - \psi)(x_n , T) \to 0$. If $T = T_\gamma$, it would imply for $n \in \mathbb{N} $ large enough that $(w^\varepsilon - \psi)(x_n , T) < \frac{\gamma e^{\frac{\varepsilon}{T}}}{2}$ and a contradiction follows from
   \[ 0 < (w^\varepsilon - \psi_\gamma)(x_n , T) < -\frac{\gamma e^{\frac{\varepsilon}{T}}}{2}.\]
   We deduce the existence of $\tau \in\, ]T_\gamma, T[$ and $x_\tau>0$ such that 
   \[  \psi_\gamma(x_\tau, \tau) < w^\varepsilon(x_\tau, \tau ) . \]
   Finally, we introduce 
   \[\begin{aligned}
       &\psi_{\gamma, \sigma} := \psi_\gamma  + \sigma (x- x_\tau)^2  \\
       \text{and } \quad &  T_\sigma := \inf \left\lbrace t>0: \quad \exists x > 0, \  w^\varepsilon (x,t) > \psi_{\gamma, \tau}(x,t) \right\rbrace.
   \end{aligned} \]
   We underline that $T_1 \leq T_\gamma  \leq  T_\sigma  \leq \tau < T $ since $ \psi_{\gamma, \sigma}(x_\tau,\tau)< w^\varepsilon (x_\tau,\tau) $. \\
   Moreover, for all $x>0$ such that $|x - x_\tau| > \sqrt{\frac{\gamma e^{-\frac{1}{T_\sigma}}}{\sigma}}$, one has that $w^\varepsilon (x, T_\sigma) \leq \psi(x, T_\sigma)\leq \psi_{\gamma, \sigma} (x , T_\sigma)$ since $T_\sigma < T$. We deduce that there exists $x_0 \in B(x_\tau,\sqrt{\frac{\gamma e^{-\frac{1}{T_\sigma}}}{\sigma}}) $ such that 
   \[0 =( w^\varepsilon - \psi_{\gamma, \sigma} )(x_0, T_\sigma) = \max ( w^\varepsilon - \psi_{\gamma, \sigma} ).\]
   As above, we deduce that 
   \[ B -\frac{\gamma e^{-\frac{1}{T_\sigma}}}{T_\sigma^2} - [A' + 2\sigma (x_0 - x_\tau)]^2 - \varepsilon  2 \sigma  \leq C_R + C_\delta.\]
   Next, using the bounds on $|x_0 - x_\tau|$ and $T_\sigma$, we conclude that 
   \[ B -\frac{\gamma}{T_1^2} - [A' + 2\sqrt{\gamma \sigma}]^2 - \varepsilon 2 \sigma \leq B -\frac{\gamma e^{-\frac{1}{T_\sigma}}}{T_\sigma^2} - [A' + 2\sigma (x_0 - x_\tau)]^2 - \varepsilon  2 \sigma \leq C_R + C_\delta   .\]
   Passing to the inferior limits $\sigma \to 0$ and then $\gamma \to 0$, it follows 
   \[ B - A'^2 \leq C_R + C_\delta\]
   which is absurd for $B> C_R + 1 + A'^2$. 
\end{itemize}
It concludes the proof of the upper bound. 

\vspace{0.5cm}

\textbf{The lower bound. } First, we define $\phi(x,t) = - a'x^2-bt  - c$ with $a',b>0$ two free parameters satisfying $a <  a'$. We prove the lower bound for $u^\varepsilon_i$ with $i \in \left\lbrace 1,2 \right\rbrace$. As above, we introduce 
\[ T := \inf \left\lbrace t>0 :  \quad \exists x>0 , \ u_i^\varepsilon(x,t) < \phi(x,t) \right\rbrace.\]
Remarking that $\phi(x, t=0) \leq u^\varepsilon_i(x,t=0)$, we deduce that $T>0$. As for the upper bound, we distinguish the proof into two cases:
\begin{itemize}
    \item \textbf{Case 1 :} \textit{There exists $x_0 >0$ such that $u^\varepsilon_i(x_0, T) = \phi(x_0, T)$. } In this case, we have 
    \[ \partial_t (u_1^\varepsilon - \phi)(x_0, T) \leq 0, \quad - \partial_{xx} (u_1^\varepsilon - \phi)(x_0, T) \leq 0 \quad \text{ and } \quad   \partial_x (u_1^\varepsilon - \phi)(x_0, T) = 0.\]
    We deduce that 
    \[-b -(2a'x_0)^2   \leq  \left(\partial_t \phi - \varepsilon  d_i \partial_{xx} \phi -  d_i | \partial_x \phi|^2 \right)(x_0, T) \\ \geq  \left(\partial_t u_1^\varepsilon - \varepsilon \partial_{xx} u_1^\varepsilon -    |\partial_x u_1^\varepsilon|^2 \right)(x_0, T) \geq - C_R.\]
    It is impossible for $a',b$ large enough (the first above inequality is strict if and only if $d_i = 0$). 
    \item \textbf{Case 2 : }\textit{There holds $u_i^\varepsilon(x, T) > \phi(x, T)$  for all $x>0$. } As for the upper bound, we introduce for $\gamma \in (0, 1]$
    \[ \begin{aligned}
        &\phi_\gamma = \phi + \gamma e^{-\frac{\varepsilon}{t}}\\
        \text{ and } \quad & T_\gamma :=\inf \left\lbrace t>0 :  \quad \exists x>0 , \ u_1^\varepsilon(x,t) < \phi_\gamma(x,t) \right\rbrace.
    \end{aligned}\]
    It is clear that $0 < T_1 \leq T_\gamma < T$. Next, there exists $\tau \in ]T_\gamma, T[$ and $x_\tau > 0$ such that 
    \[ \phi_\gamma(x_\tau, \tau) > u_1^\varepsilon(x_\tau, \tau). \]
    We introduce 
        \[ \begin{aligned}
        &\phi_{\gamma, \sigma} = \phi_\gamma - \sigma (x - x_\tau)^2\\
        \text{ and } \quad & T_\sigma :=\inf \left\lbrace t>0 :  \quad \exists x>0 , \ u_i^\varepsilon(x,t) < \phi_{\gamma, \sigma}(x,t) \right\rbrace.
    \end{aligned}\]
    Moreover, we have $0 < T_1 \leq T_\gamma < T_\sigma < \tau \leq T$. Since for $x \in B^c(x_\tau, \sqrt{\frac{\gamma}{\sigma}})$, we have $\phi_{\gamma, \sigma} (x, T_\sigma)< \phi(x, T_\sigma) < u_i^\varepsilon (x, T_\sigma)  $ (since $T_\sigma < \tau)$, it follows the existence of $x_0 \in B(x_\tau, \sqrt{\frac{\gamma}{\sigma}})$ such that 
    \[ u_i^\varepsilon (x_0, T_\sigma) = \phi_{\gamma, \sigma} (x_0, T_\sigma).\]
    A direct computation implies 
    \[ \begin{aligned}
    -b +\frac{\gamma}{T_1^2} - [2a'x_0 - 2\sqrt{\gamma \sigma}]^2 + \varepsilon2  (\sigma ) &\geq \left(\partial_t \phi - \varepsilon  d_i\partial_{xx} \phi -   d_i[\partial_x \phi]^2 \right) (x_0, T) \\
    &\geq   \left(\partial_t u_i^\varepsilon - \varepsilon d_i \partial_{xx} u_i^\varepsilon - d_i  [\partial_x u_i^\varepsilon]^2 \right) (x_0, T) \\
    &\geq -C_R   .       
    \end{aligned} \]
Taking the inferior limits $\sigma \to 0$ and $\gamma \to 0$ and $b$ large enough, leads to the desired contradiction. 
\end{itemize}
It concludes the proof.
  \flushright\qed
\end{proof}

\subsection{A first weak asymptotic result for $u_i^\varepsilon - u_j^\varepsilon$}

As mentioned in the comment that follows the statement of Theorem\ref{HJ:thm:interm_u}, we only prove a first imprecise (but necessary) result on $u_i^\varepsilon - u_j^\varepsilon$. For this purpose, we introduce 

\begin{definition}
Let $q_i^+$ be defined by
\begin{equation}
    \label{HJ:app:eq:q}
    q_i^+ =\left\lbrace
    \begin{aligned}
    &\frac{  \sqrt{([d_i - d_j ] (\partial_x u_j^\varepsilon)^2 + r_i - r_j)^2 + 4 \delta_i \delta_j}}{2 \delta_j } && \quad \text{ when } d_i < d_j, \\
    &q_i && \quad \text{ when } d_i \geq d_j.
    \end{aligned} 
    \right.
\end{equation}
\end{definition}
We emphasize that 
\begin{equation}
    \label{HJ:App:eq:qandq+}
    q_i(x,t) \leq q_i^+(x,t) \qquad \forall (x,t) \in \mathbb{R}^+ \times \mathbb{R}^+. 
\end{equation}  
This new quantity is introduced in order to prove the following result
\begin{lemma}\label{HJ:App:lemma}
Under the hypothesis \eqref{HJ:Hyp:Rdelta}, we have 
\[ - c (x+1)\leq \ln(q_1^+(x,t)).\]
\end{lemma}

\begin{proof}
By definition of $q_1^+$, in any case and for all $\varepsilon>0$, we have 
\[\frac{\ln(\delta_1(x))-\ln(\delta_2(x))}{2} \leq \ln\left(\sqrt{\frac{\delta_1(x)}{\delta_2(x)}}\right)  \leq \ln( q_1^+(x,t) )   . \]
Thanks to \eqref{HJ:Hyp:Rdelta}, we have 
\[ -C_\delta x - \ln(C_\delta) \leq \ln( q_1^+(x,t) ). \]
\flushright\qed
\end{proof}
 Notice that when $d_1<d_2$, the conclusion of Lemma \ref{HJ:App:lemma} may be false for $q$ if $|\partial_x u_2^\varepsilon|$ is not locally bounded. With this result, one can state the following lemma:
 
 \begin{lemma}\label{HJ:lemma:weakAsympt}
 Under the hypothesis \eqref{HJ:Hyp:di}, \eqref{HJ:Hyp:Rdelta},  \eqref{HJ:Hyp:R} and \eqref{HJ:Hyp:Init_cond}, we have 
 \[ \varepsilon \left( \ln (q_2^+(x,t)) - \frac{\varepsilon^4 }{t} \right) \leq (u_1^\varepsilon - u_2^\varepsilon)(x,t) \leq \varepsilon \left( \ln (q_1^+(x,t)) + \frac{\varepsilon^4 }{t} \right) . \]
 \end{lemma}

\begin{proof}
Set $(x_0, t_0) \in \mathbb{R}^+\times \mathbb{R}^+$, $\varepsilon >0$ and $\mu>0$. Next, we introduce
\[ \tau_\mu := \inf \left\lbrace t > 0: \quad \exists x>0 \quad\text{ such that } \quad  (u_1^\varepsilon - u_2^\varepsilon )(x,t) - \mu^{-1} (x-x_0)^2 > \varepsilon\left[ \ln (q_1^+(x,t)) + \frac{\varepsilon^4}{t}\right] \right\rbrace.\]
Thanks to 1. of Theorem\ref{HJ:thm:interm_u}, we have $\tau_\mu>0$. Remark also that for all $\mu<1$, we have  $\forall (x,t) \in \mathbb{R}^+ \times [0, \tau_1]$
\[  (u_1^\varepsilon - u_2^\varepsilon )(x,t) - \mu^{-1} (x-x_0)^2  =   (u_1^\varepsilon - u_2^\varepsilon )(x,t) - (x-x_0)^2 + [\mu^{-1}-1] (x-x_0)^2<\varepsilon \left[ \ln q_1(x,t) + \frac{\varepsilon^4}{t} \right] .  \] 
It follows that $\tau_\mu>\tau_1>0$ for all $\mu<1$. Next, we distinguish two cases:
\begin{enumerate}
    \item $\underset{\mu \rightarrow 0}{\limsup}  \ \tau_\mu > t_0$, 
    \item $\underset{\mu \rightarrow 0}{\liminf} \  \tau_\mu \leq t_0$. 
\end{enumerate}
We will only consider the second case, since it is clear that in the first case the conclusion holds true. 

\vspace{0.2cm}

We prove by contradiction that this case can not hold. Let $\mu_n \to 0$ be such that $\tau_{\mu_n}$ converges to $\underset{\mu \rightarrow 0}{\liminf} \tau_\mu$. For sake of readability, we replace $\tau_{\mu_n}$ by $\tau_\mu$. Notice that this limit belongs to $[\tau_1, t_0]$. \\
According to the point 1. of Theorem\ref{HJ:thm:interm_u} and Lemma \ref{HJ:App:lemma}, we have
\[(u_1^\varepsilon - u_2^\varepsilon )(x, \tau_\mu) -\mu^{-1}(x-x_0)^2 - \varepsilon \left[\ln (q_1^+(x,\tau_\mu)) + \frac{\varepsilon^4}{\tau_\mu} \right] \underset{ x \to +\infty}{\longrightarrow} -\infty.\]
It follows the existence of $x_\mu>0$ such that 
\[ \begin{aligned}
0 \  =& \ (u_1^\varepsilon - u_2^\varepsilon )(x_\mu, \tau_\mu)   - \mu^{-1}(x_\mu-x_0)^2- \varepsilon \left[ \ln (q_1^+(x_\mu, \tau_\mu)) + \frac{\varepsilon^4}{\tau_\mu} \right]\\
  =& \ \underset{ x >0}{\max} \  (u_1^\varepsilon - u_2^\varepsilon )(x, \tau_{\mu}) - \mu^{-1} (x-x_0)^2 - \varepsilon \left[ \ln (q_1^+(x,\tau_\mu)) + \frac{\varepsilon^4}{\tau_\mu} \right].
\end{aligned}\]
Moreover, we observe that as $\mu \to 0$, one has $x_\mu \to x_0$. One also has 
\begin{equation}\label{eq:bounds_u_v1}
\begin{aligned}
    &\partial_t (u_1^\varepsilon - u_2^\varepsilon )(x_\mu, \tau_\mu) = -\frac{\varepsilon^4}{\tau_\mu^2} +  \varepsilon \partial_t \left(\ln (q_1^+(x_\mu, \tau_\mu) \right), \\
    &\partial_x(u_1^\varepsilon - u_2^\varepsilon )(x_\mu, \tau_\mu) = 2\mu^{-1} (x_\mu - x_0) +  \varepsilon \partial_x ( \ln(q_1^+(x_\mu, \tau_\mu))  \\
    \text{ and } \qquad -&\partial_{xx} (u_1^\varepsilon - u_2^\varepsilon )(x_\mu, \tau_\mu)  \geq  2\mu^{-1} +   \varepsilon( \partial_{xx}\ln(q_1^+(x_\mu, \tau_\mu)).
    \end{aligned}
\end{equation}
Since $(x_\mu , \tau_\mu)$ converges as $\mu \to 0$ and all the involved functions are continuous, we deduce the existence of $C>0$ (independent of $\mu$ but that may depend on $\varepsilon$) such that for all $\mu>0$ small enough,
\begin{equation} \label{eq:bounds_u_v2}
\begin{aligned}
\max\Big(&| \frac{\varepsilon^4}{\tau_\mu^2}  +   \partial_t \left(\ln (q_1^+(x_\mu, \tau_\mu) \right) |, |  \partial_x ( \ln(q_1^+(x_\mu, \tau_\mu))| ,  \\
& \qquad | \partial_{xx}( \ln(q_1^+(x_\mu, \tau_\mu))| ,|\partial_x (u_1^\varepsilon + u_2^\varepsilon) (x_\mu, t_\mu)|, |\partial_{xx} u_2^\varepsilon (x_\mu, \tau_\mu)|\Big) < C.  
\end{aligned}
\end{equation}
We subtract the equations for $u_1^\varepsilon$ and $u_2^\varepsilon$ and we obtain
\begin{equation} \label{eq:diff_u_v}   
\begin{aligned}
&\partial_t (u_1^\varepsilon - u_2^\varepsilon) - d_1\partial_x(u_1^\varepsilon - u_2^\varepsilon) \partial_x(u_1^\varepsilon + u_2^\varepsilon ) - d_1\varepsilon \partial_{xx} (u_1^\varepsilon - u_2^\varepsilon )  -  [d_1 + d_2] \varepsilon  \partial_{xx} u_2^\varepsilon \\
= \quad& [d_1 - d_2] (\partial_x u_2^\varepsilon)^2 + r_1 - r_2 + \delta_1 e^{\frac{u_2^\varepsilon - u_1^\varepsilon }{\varepsilon} }- \delta_2 e^{\frac{u_1^\varepsilon - u_2^\varepsilon }{\varepsilon} }\\
= \quad& e^{\frac{u_2^\varepsilon - u_1^\varepsilon}{\varepsilon }} P_{d_1,d_2}(e^{\frac{u_1^\varepsilon - u_2^\varepsilon }{\varepsilon}}). \\
\end{aligned}
\end{equation}
Next, we evaluate the above equation at $(x_\mu, \tau_\mu)$. First, since $(u_1^\varepsilon - u_2^\varepsilon )(x_\mu, \tau_\mu) = \varepsilon \ln (q_1^+(x_\mu, \tau_\mu)) + \mu^{-1}(x_\mu-x_0)^2 +\frac{ \varepsilon^2}{\tau_\mu}$, we deduce that 
\[q_1(x_\mu, \tau_\mu) \leq q_1^+(x_\mu, \tau_\mu) < e^{\frac{(u_1^\varepsilon - u_2^\varepsilon )(x_\mu, \tau_\mu)}{\varepsilon}}.\]
It follows
\[  \left( e^{\frac{u_2^\varepsilon - u_1^\varepsilon}{\varepsilon }} P_{d_1,d_2}(e^{\frac{u_1^\varepsilon - u_2^\varepsilon }{\varepsilon}}) \right) (x_\mu, \tau_\mu) \leq  0. \]
We deduce thanks to \eqref{eq:bounds_u_v1}, \eqref{eq:bounds_u_v2} and \eqref{eq:diff_u_v} that there holds for $\mu$ small enough
\[ \begin{aligned}
  0&<\quad C + C[2\mu^{-1} |x_\mu - x_0| + C ] + 2 \varepsilon \mu^{-1} + \varepsilon C (1 + d_1 + d_2)  \\
   &=\quad\left(\partial_t (u_1^\varepsilon - u_2^\varepsilon) - d_1 \partial_x(u_1^\varepsilon - u_2^\varepsilon) \partial_x(u_1^\varepsilon + u_2^\varepsilon ) - \varepsilon d_1 \partial_{xx} (u_1^\varepsilon - u_2^\varepsilon )  \right) (x_\mu , t_\mu)   -  [d_1 + d_2] \varepsilon  \partial_{xx} u_2^\varepsilon(x_\mu, \tau_\mu)\\
   &= \quad \left( e^{\frac{u_2^\varepsilon - u_1^\varepsilon}{\varepsilon }} P_{d_1, d_2}(e^{\frac{u_1^\varepsilon - u_2^\varepsilon }{\varepsilon}}) \right) (x_\mu, \tau_\mu) \\
    &\leq\quad  0.
\end{aligned} \]
We have reached the desired contradiction.

\vspace{0.2cm}

The proof for $q_2^+$ is identical by studying $u_2^\varepsilon - u_1^\varepsilon$. Therefore, we let it to the reader. 

\flushright\qed
\end{proof}

It follows the following corollary
\begin{corollary}\label{HJ:corol1}
 Under the hypothesis \eqref{HJ:Hyp:di}, \eqref{HJ:Hyp:Rdelta}, \eqref{HJ:Hyp:Init_cond} and \eqref{HJ:Hyp:R} we have 
 \begin{equation}
   \frac{n_i^\varepsilon(x,t) }{n_j^\varepsilon(x,t)} =   e^\frac{(u_i^\varepsilon - u_j^\varepsilon)(x,t)}{\varepsilon} \leq C e^{C_\delta x + \frac{\varepsilon^3}{t} } [(\partial_x u_j^\varepsilon (x,t)^2 + 1] .
 \end{equation}
\end{corollary}

\begin{proof}
We focus on the case $i=1, \ j=2$, the other case works exactly the same. According to Lemma \ref{HJ:lemma:weakAsympt}, it is sufficient to prove that 
\[ q_1^+ (x,t) \leq C e^{C_\delta x  } [\partial_x u_2^\varepsilon (x,t)^2 + 1] . \]
First, we remark that thanks to \eqref{HJ:Hyp:Rdelta}
\[ \frac{1}{\delta_1(x)} \leq Ce^{C_\delta x}. \]
Next, we treat the numerator of $q_1^+$. When $d_1<d_2$, we have 
\[ \begin{aligned}
 \sqrt{\left((d_1-d_2) (\partial_x u_2^\varepsilon)^2 + (r_1-r_2)\right)^2 + 4 \delta_1 \delta_2 }&\leq  (\partial_x u_2^\varepsilon)^2 \sqrt{\left((d_1-d_2) + \frac{(r_1-r_2)}{(\partial_x u_2^\varepsilon)^2}\right)^2 + \frac{4 \delta_1 \delta_2}{(\partial_x u_2^\varepsilon)^4} } \\
 &\leq  (\partial_x u_2^\varepsilon)^2 \sqrt{\left((d_1 +d_2 + \frac{2 C_R}{(\partial_x u_2^\varepsilon)^2}\right)^2 + \frac{4 C_\delta^2}{(\partial_x u_2^\varepsilon)^4} } \\
 &\leq  C[(\partial_x u_2^\varepsilon)^2+ 1].
\end{aligned}\]
Combining the two above equations the conclusion follows. 

\vspace{0.2cm}

For the case, $d_1\geq d_2$, we simply have thanks to the above computations
\[  q_1^+ (x,t) \leq   \left[(\partial_x u_2)^2 (C + \frac{(d_1 -d_2)}{2}+ 1) + C_R \right] e^{C_\delta x } \leq C e^{C_\delta x} [(\partial_x u_2^\varepsilon (x,t)^2 + 1] .\]
\flushright\qed
\end{proof}

\subsection{Space regularity of $u_i^\varepsilon$}\label{App:secB2}

\begin{proof}[Proof of 2. of Theorem\ref{HJ:thm:interm_u}]
First, we fix an initial time $t_1>0$ and a maximal time $T>0$. Next, we define $u_i^\varepsilon = f(v_i^\varepsilon )$ where $v_i^\varepsilon$ will be chosen later on. A direct computation yields:
\[ \partial_t u_i^\varepsilon  = f'(v_i^\varepsilon ) \partial_t v_i^\varepsilon, \quad  \partial_x u_i^\varepsilon  = f'(v_i^\varepsilon )  \partial_x v_i^\varepsilon \quad \text{ and } \quad \partial_{xx} u_i^\varepsilon = f'(v_i^\varepsilon) \partial_{xx} v_i^\varepsilon + f''(v_i^\varepsilon)[  \partial_x v_i^\varepsilon]^2.\]
Replacing in the $i^{th}$ equation \eqref{HJ:eq:main}, it follows
\begin{equation}\label{HJ:pf_Lip:eq1}  \partial_t v_i^\varepsilon - \varepsilon  d_i \partial_{xx} v_i^\varepsilon - d_i \left( \varepsilon \frac{ f''(v_i^\varepsilon)}{f'(v_i^\varepsilon)} + f'(v_i^\varepsilon) \right) [\partial_x v_i^\varepsilon ]^2= \frac{r_i(x) - N(t) + \delta_i(x) e^{\frac{u_j^\varepsilon - u_i^\varepsilon}{\varepsilon}}}{f'(v_i^\varepsilon)} .\end{equation}
Next, we differentiate \eqref{HJ:pf_Lip:eq1} with respect to $x$ and we multiply by $\partial_x v_i^\varepsilon$ to obtain:
\[\begin{aligned}
&\partial_t ([\partial_x v_i^\varepsilon]^2) -\varepsilon d_i \partial_{xx} ([\partial_x v_i^\varepsilon]^2) + 2\varepsilon d_i (\partial_{xx} v_i^\varepsilon)^2-4 d_i \left( \frac{f''(v_i^\varepsilon)}{f'(v_i^\varepsilon)} + f'(v_i^\varepsilon)  \right) \partial_x{ ([\partial_x v_i^\varepsilon]^2)}  \partial_x v_i^\varepsilon \\
&- 2d_i\left(    \varepsilon \frac{f'''(v_i^\varepsilon)}{f'(v_i^\varepsilon ) }- \varepsilon \frac{f''(v_i^\varepsilon)^2}{f'(v_i^\varepsilon)} + f''(v_i^\varepsilon) \right)  [\partial_x v_i^\varepsilon]^4  =-2\frac{f''(v_i^\varepsilon)}{f'(v_i^\varepsilon)}r_i \partial_x v_i^\varepsilon +2 \frac{\partial_x r_i}{f'(v_i^\varepsilon)} \partial_x v_i^\varepsilon \\
\quad & +2e^{\frac{f(v_j^\varepsilon) - f(v_i^\varepsilon) }{\varepsilon}} \left(\frac{\delta_i' f'(v_i^\varepsilon) - \delta_i f''(v_i^\varepsilon) }{f'(v_i^\varepsilon)^2}  \partial_x v_i^\varepsilon + \delta_i \left[ \frac{\partial_x v_j^\varepsilon f'(v_j^\varepsilon) \partial_x v_i^\varepsilon f'(v_i^\varepsilon) - (\partial_x v_i^\varepsilon f'(v_i^\varepsilon))^2 }{\varepsilon f'(v_i^\varepsilon)^2} \right]\right).
\end{aligned}\]
Next, we assume that $(\partial_x u_i^\varepsilon)^2 = \max ( (\partial_x u_1^\varepsilon)^2, (\partial_x u_2^\varepsilon)^2)$. It follows \[ \frac{\partial_x v_j^\varepsilon f'(v_j^\varepsilon) \partial_x v_i^\varepsilon f'(v_i^\varepsilon) - (\partial_x v_i^\varepsilon f'(v_i^\varepsilon))^2 }{\varepsilon f'(v_i^\varepsilon)^2} \leq 0.\] Next, by defining $f(v) = 2D - v^2$ where $D$ is large enough such that $f(v)>D$ (thanks to point 1. of Theorem\ref{HJ:thm:interm_u}) and dividing by $|\partial_x v_i^\varepsilon|$, we obtain thanks to Corollary \ref{HJ:corol1} 
\[ \begin{aligned}
&\partial_t (|\partial_x v_i^\varepsilon|) -\varepsilon \partial_{xx} (|\partial_x v_i^\varepsilon|)-4\left( \frac{f''(v_i^\varepsilon)}{f'(v_i^\varepsilon)} + f'(v_i^\varepsilon)  \right) \partial_x{ ([\partial_x v_i^\varepsilon]^2)}  \partial_x v_i^\varepsilon +4 [\partial_x v_i^\varepsilon]^3  -4\frac{C_r}{D}\\
 \leq & \ 2e^{\frac{f(v_j^\varepsilon) - f(v_i^\varepsilon) }{\varepsilon}} \left(\frac{\delta_i' f'(v_i^\varepsilon) - \delta_i f''(v_i^\varepsilon) }{f'(v_i^\varepsilon)^2}  \right) \\
  \leq & \ \frac{2Ce^{\frac{\varepsilon^4}{t_1}}(|\partial_x u_i^\varepsilon|^2 + 1)}{\delta_i} \max\left(\frac{\delta_i' f'(v_i^\varepsilon) - \delta_i f''(v_i^\varepsilon) }{f'(v_i^\varepsilon)^2}  , 0\right)  .
\end{aligned} \]
Next, thanks to Corollary \ref{HJ:corol1}, for any $R_0>0$, it follows the existence of a large constant $C(t_1, T)$ (independent of $\varepsilon$), such that for 
\[  \theta(T, R_0 )  = C(t_1, T) e^{C_\delta R_0}  \]
we have
\[ \partial_t (|\partial_x v_i^\varepsilon|) -\varepsilon \partial_{xx} (|\partial_x v_i^\varepsilon|)-4\left( \frac{f''(v_i^\varepsilon)}{f'(v_i^\varepsilon)} + f'(v_i^\varepsilon)  \right) \partial_x{ ([\partial_x v_i^\varepsilon]^2)}  \partial_x v_i^\varepsilon +[|\partial_x v_i^\varepsilon  | - \theta(T,R_0)]^3< 0\]
Following, the Appendix B of \cite{BMP}, the conclusion follows by comparing $|\partial_x v_i^\varepsilon|$ to $\frac{1}{2\sqrt{t}} + \theta(T,R)$. We conclude that
\[  \max(|\partial_x u_1^\varepsilon|, |\partial_x u_2^\varepsilon|)(x,t) \leq  \frac{C(t_1, T) e^{C_\delta R}}{\sqrt{t_1}} \qquad \text{ for } \quad (x,t) \in [0, R] \times [t_1, T]. \]
  \flushright\qed
\end{proof}

\begin{corollary}\label{HJ:Corol2}
Under the hypotheses \eqref{HJ:Hyp:di}, \eqref{HJ:Hyp:Rdelta}, \eqref{HJ:Hyp:R} and \eqref{HJ:Hyp:Init_cond} we have 
\[ \max(|\partial_x u_1^\varepsilon|, |\partial_x u_2^\varepsilon|)(x,t) \leq  C(t_1, T) e^{C_\delta x} \qquad \text{ for } \quad (x,t) \in \mathbb{R}^+ \times [t_1, T] \]
(where $C(t_1, T)$ is a new arbitrary large constants that depends only on $t_1$ and $T$). 
\end{corollary}


\subsection{Asymptotic of $u_i^\varepsilon - u_j^\varepsilon$}

We only prove the following Lemma

\begin{lemma}\label{HJ:Lemma:q}
 Under the hypotheses \eqref{HJ:Hyp:di}, \eqref{HJ:Hyp:Rdelta}, \eqref{HJ:Hyp:R} and \eqref{HJ:Hyp:Init_cond}, for any time interval $[t_1, T]$, there exists $C(t_1,T)> 0$ such that 
\[ -C(t_1, T)(x+1) \leq \ln (q_i(x, t)) < C(t_1,T)(x+1).\]
\end{lemma}

Indeed, it is sufficient to prove this lemma because the proof of the upper bound point 3. of Theorem\ref{HJ:thm:interm_u} is the same than the proof of Lemma \ref{HJ:lemma:weakAsympt} by replacing $q_i^+$ by $q_i$ and by using the lower bound provided by Lemma \ref{HJ:Lemma:q} instead of the estimate provided by Lemma \ref{HJ:App:lemma}. Notice that the proof of the lower bounds follows exactly the same argument than the upper bound except that $P_{d_i,d_j}\left(e^\frac{(u_i-u_j)(x_\mu), \tau_\mu}{\varepsilon}\right)>0$. Therefore, we let the details of the proof for the reader. 

\vspace{0.2cm}

\begin{proof}[Proof of Lemma \ref{HJ:Lemma:q}]
We prove this lemma for $i=1$, the proof works the same with $i=2$. We underline that the constant $C(t_1,T)$ can increase from line to line but does not depend on $x$ or $\varepsilon$. 

\vspace{0.2cm}

$\bullet$\textbf{ The upper bound. } We start from the definition of $q_1$: 
\begin{equation}\label{HJ:eq:pfq0}
\begin{aligned}
\ln( q_1(x,t) ) = &\ln\left( [d_1 -d_2](\partial_x u_2^\varepsilon(x,t))^2 + r_1-r_2 + \sqrt{ ( [d_1 -d_2](\partial_x u_2^\varepsilon(x,t))^2 + r_1-r_2 )^2+ 4\delta_1 \delta_2 } \right) \\
& \qquad- \ln( 2 \delta_2). 
\end{aligned}
\end{equation}
According to Corollary \ref{HJ:Corol2}, we have that for all $t \in [t_1, T]$
\begin{equation*}\label{HJ:eq:pfq1}
\begin{aligned}
    &[d_1 -d_2](\partial_x u_2^\varepsilon(x,t))^2 + r_1-r_2 + \sqrt{ ( [d_1 -d_2](\partial_x u_2^\varepsilon(x,t))^2 + r_1-r_2 )^2+ 4\delta_1 \delta_2 } \\
    \leq \ & (C(t_1, T) e^{C_\delta x})^2 \left( d_1 +d_2 +  \sqrt{ ( [d_1 -d_2](1 + \frac{r_1-r_2 }{ e^{2C_\delta x}})^2+ \frac{4\delta_1 \delta_2}{ e^{4C_\delta x}} } \right) + 2 C_R \\
    \leq \ & C(t_1,T) (e^{2C_\delta x} + 1).
\end{aligned}
\end{equation*}
It follows 
\begin{equation}\label{HJ:eq:pfq1}
\begin{aligned}
    &\ln\left([d_1 -d_2](\partial_x u_2^\varepsilon(x,t))^2 + r_1-r_2 + \sqrt{ ( [d_1 -d_2](\partial_x u_2^\varepsilon(x,t))^2 + r_1-r_2 )^2+ 4\delta_1 \delta_2 } \right)\\
   & \leq  [2 C_\delta x +  C(t_1, T)].
\end{aligned}
\end{equation}
Thanks to \eqref{HJ:Hyp:Rdelta}, we have 
\begin{equation}\label{HJ:eq:pfq2}
     -\ln( 2 \delta_2) \leq C_\delta x + C. 
\end{equation}
Inserting \eqref{HJ:eq:pfq1} and \eqref{HJ:eq:pfq2} into \eqref{HJ:eq:pfq0}, the conclusion follows for the upper bound. 

\vspace{0.2cm}

$\bullet$\textbf{ The lower bound. } If $d_1 \geq d_2$, then the result is exactly the one obtained in Lemma \ref{HJ:App:lemma}. Therefore, we only consider the case $d_1<d_2$, in this case we have 
\begin{align*}
 q_1(x,t)  = &\frac{[d_1 -d_2](\partial_x u_2^\varepsilon(x,t))^2 + r_1-r_2 + \sqrt{ ( [d_1 -d_2](\partial_x u_2^\varepsilon(x,t))^2 + r_1-r_2 )^2+ 4\delta_1 \delta_2 } }{ 2 \delta_2} \\
 = & \frac{ 2 \delta_1}{[d_2 -d_1](\partial_x u_2^\varepsilon(x,t))^2 + r_2-r_1 + \sqrt{ ( [d_1 -d_2](\partial_x u_2^\varepsilon(x,t))^2 + r_1-r_2 )^2+ 4\delta_1 \delta_2 } }.
 \end{align*}
 Next, following similar computations than \eqref{HJ:eq:pfq1} and \eqref{HJ:eq:pfq2} the conclusion follows for the lower bound.
  \flushright\qed
\end{proof}

To finish, we state a Proposition that provides some identity related to $q_i$. The proposition follows from straightforward computations that we omit here. However, the following identities will be very useful in the proof of point 1. of Theorem\ref{HJ:thm:main_thm}. 

\begin{proposition}\label{HJ:propo:identity}
The following identities hold true:
\begin{enumerate}
    \item $ r_j + \delta_j q_i = r_H^D(\partial_x u_j)$ where $r_H^D$ is introduced in \eqref{HJ:def:rh}, 
    \item $q_i^{-1} = \frac{[d_j - d_i] (\partial_x u_j)^2 + r_j - r_i + \sqrt{([d_j - d_i] (\partial_x u_j)^2 + r_j - r_i)^2 + 4\delta_i \delta_j }}{2 \delta_i} $, 
    \item $r_i + \delta_i q_i^{-1} = r_H^D(\partial_x u_j)$.
\end{enumerate}
\end{proposition}

Notice that in the special case $d_i = d_j$, we recover $q_i^{-1} = q_j$. 

\subsection{Time regularity of $u_i^\varepsilon$ }

The local Lipschitz time regularity of $u_i^\varepsilon$ is an exact transposition of the proof of the time regularity of $u_\varepsilon$ in \cite{BMP} in Section 3.5. It is performed with the so-called method of doubling variable and relies mainly on the above bounds about $\partial_x u_i^\varepsilon$. We do not provide this proof and just refer to \cite{BMP} Section 3.5.

\section{The Hamilton Jacobi convergence result}\label{Sec4:HJ}


\begin{proof}[Proof of Theorem\ref{HJ:thm:main_thm}]
We split the proof in several parts:
\begin{enumerate}
    \item The convergence of $N_\varepsilon$,  
    \item The convergence of $u_i^\varepsilon$ to $u$ and the control condition, 
    \item The function $u$ is solution of \eqref{HJ:eq:main_eq}, 
    \item The convergence of $n_\varepsilon$. 
\end{enumerate}

\textbf{$\bullet$ Convergence of $N_\varepsilon.$ } We follow the proofs of Theorem 3.1 of \cite{PB} and Theorem 2.4 of \cite{BMP}. 
First, we sum the two equations and we integrate with respect to $x$, it follows 
\[ N_\varepsilon'(t)  = \frac{1}{\varepsilon } \int_{0}^{+\infty}(1 \ 1 )\textbf{ R}(x, N_\varepsilon(t)) \textbf{n}_\varepsilon(x,t) dx  : = J_\varepsilon(t)     .\]
Notice that for all $t>0$, we have 
\begin{equation}\label{HJ:eq:J}
J_\varepsilon (t) \leq \frac{1}{\varepsilon } C_R C_N. 
\end{equation}
Next, we differentiate $J_\varepsilon$ over time and it follows thanks to \eqref{HJ:Hyp:R}
\[\begin{aligned}
 J_\varepsilon'(t) & = \frac{J_\varepsilon(t) }{\varepsilon } \int_0^{+\infty} (1 \ 1 ) \  \partial_2 \textbf{R}(x, N_\varepsilon) \textbf{n}^\varepsilon(x,t)dx + \frac{1}{\varepsilon^2}  \int_{0}^{+\infty}(1 \ 1 )\textbf{ R}(x, N_\varepsilon(t))^2 \textbf{n}^\varepsilon(x,t) dx \\
 & +  \frac{1}{\varepsilon^2}  \int_{0}^{+\infty}(1 \ 1 )\textbf{ R}(x, N_\varepsilon(t)) \textbf{D} \partial_{xx} \textbf{n}^\varepsilon(x,t) dx \\
  & = - \frac{N_\varepsilon(t) J_\varepsilon(t) }{\varepsilon } + \int_{0}^{+\infty}(1 \ 1 )\textbf{ R}(x, N_\varepsilon(t))^2 \textbf{n}^\varepsilon(x,t) dx+ \frac{1}{\varepsilon^2}  \int_{0}^{+\infty}(1 \ 1 )\textbf{ R}(x, N_\varepsilon(t)) \textbf{D} \partial_{xx} \textbf{n}^\varepsilon(x,t) dx \\
  & \geq -C_N \left(C_R +   \frac{J_\varepsilon}{\varepsilon} \right) + \frac{1}{\varepsilon^2}  \int_{0}^{+\infty}(1 \ 1 )\textbf{ R}(x, N_\varepsilon(t))^2 \textbf{n}^\varepsilon(x,t) dx .
\end{aligned}\]
As mentioned in the introduction, a new technical difficulty arises since we deal with a system: it is not clear that the quantity
\[  (1 \ 1 )\textbf{ R}(\cdot,  N_\varepsilon)^2 \textbf{n}^\varepsilon \geq 0.  \]
Indeed, using mainly Proposition \ref{HJ:propo:identity} on $q_i$, we prove  
\[ \frac{1}{\varepsilon} \int_{0}^{+\infty}(1 \ 1 )\textbf{ R}(x, N_\varepsilon(t))^2 \textbf{n}^\varepsilon(x,t) dx > - C \varepsilon   \]
which is enough to conclude to the convergence of $N_\varepsilon$ (as we will detail later on). To prove such an inequality, we start from
\begin{equation}\label{eq:pf:N1}
    \begin{aligned}
 (1 \ 1 )\textbf{ R}(\cdot,  N_\varepsilon)^2 \textbf{n}^\varepsilon = & n_1^\varepsilon \left[ (r_1 - N_\varepsilon) + \delta_1 \delta_2 + (r_1 + r_2 - 2N_\varepsilon) \delta_2 \right] \\
 & \qquad + n_2^\varepsilon \left[ (r_2 - N_\varepsilon) + \delta_1 \delta_2 + (r_1 + r_2 - 2N_\varepsilon) \delta_1 \right]  \\
 = & (r_1 + \delta_2 - N_\varepsilon) \left[ (r_1 - N_\varepsilon ) n_1^\varepsilon + \delta_1 n_2^\varepsilon \right] \\
 &\qquad + (r_2 + \delta_1 - N_\varepsilon) \left[ (r_2 - N_\varepsilon ) n_2^\varepsilon + \delta_2 n_1^\varepsilon \right].
\end{aligned} 
\end{equation}
Thanks to the point 3. of Theorem\ref{HJ:thm:interm_u} and Proposition \ref{HJ:propo:identity}, we have for $t\geq \varepsilon$
\begin{equation}\label{eq:pf:N2}
\begin{aligned}
  \left[ (r_2 - N_\varepsilon ) n_2^\varepsilon + \delta_2 n_1^\varepsilon \right] &= n_2^\varepsilon [ r_2 - N_\varepsilon + \delta_2 q_1 + o(\varepsilon^3)]\\
  &= n_2^\varepsilon[r_H^D(\partial_x u_2^\varepsilon) - N_\varepsilon + o(\varepsilon^3)].
\end{aligned} 
\end{equation}
With similar computations, we also have 
\begin{equation}\label{eq:pf:N3}
\left[ (r_1 - N_\varepsilon ) n_1^\varepsilon + \delta_1 n_2^\varepsilon \right] = n_1^\varepsilon \left[ r_1 + q_1^{-1} \delta_1 - N_\varepsilon \right] =  n_1^\varepsilon \left[ r_H^D(\partial_x u_2^\varepsilon ) - N_\varepsilon + o(\varepsilon^3) \right]. 
\end{equation}
Inserting \eqref{eq:pf:N2} and \eqref{eq:pf:N3} into \eqref{eq:pf:N1}, it follows
\[ (1 \ 1 )\textbf{ R}(\cdot,  N_\varepsilon)^2 \textbf{n}^\varepsilon  = [r_H^D(\partial_x u_2^\varepsilon ) - N_\varepsilon ] \left[ (r_1 + \delta_2)n_1^\varepsilon + (r_2 + \delta_1) n_2^\varepsilon - N_\varepsilon (n_1^\varepsilon + n_2^\varepsilon)  \right] + o(\varepsilon^3)(n_1^\varepsilon + n_2^\varepsilon). \]
Using again the point 3. of Theorem\ref{HJ:thm:interm_u} and Proposition \ref{HJ:propo:identity}, it follows 
\[ \begin{aligned} 
(r_1 + \delta_2 - N_\varepsilon)n_1^\varepsilon + (r_2 + \delta_1 - N_\varepsilon) n_2^\varepsilon  &=  (r_1 + \delta_2 - N_\varepsilon) (q_1 + o(\varepsilon^3) ) n_2^\varepsilon + (r_2 + \delta_1  - N_\varepsilon) n_2^\varepsilon \\
& = \left[(r_1 + \delta_2 ) q_1 + r_2 +\delta_1 - N_\varepsilon(1 +q_1) + o(\varepsilon^3) \right]n_2^\varepsilon\\
& = n_2^\varepsilon (1+q_1) \left[ \frac{r_2 + q_1 \delta_2 + q_1 (r_1 + q_1^{-1} \delta_1)  }{1 + q_1 }  - N_\varepsilon \right] + o(\varepsilon^3) n_2^\varepsilon \\
& = n_2^\varepsilon (1 + q_1) (r_H^D(\partial_x u_2^\varepsilon) - N_\varepsilon) + o(\varepsilon^3) n_2^\varepsilon. 
\end{aligned}\]
We deduce that 
\[  \begin{aligned}
 \frac{1}{\varepsilon^2} \int_0^{+\infty}(1 \ 1 )\textbf{ R}(x,  N_\varepsilon)^2 \textbf{n}^\varepsilon(x)dx  = & \frac{1}{\varepsilon^2}   \int_{0}^{+\infty} (1 + q(x)) [r_H^D(x, \partial_x u_2^\varepsilon) -  N_\varepsilon]^2 n_2^\varepsilon(x,t) dx  \\
 &\qquad + o(\varepsilon) \int_{0}^{+\infty}(n_1^\varepsilon + 2 n_2^\varepsilon)(x, t) dx \\
 \geq &- 2 C \varepsilon. 
\end{aligned}\]    

We conclude that for all $t > \varepsilon$, we have 
\[ J_\varepsilon' (t) > -C_N\left(\frac{ J_\varepsilon (t)}{\varepsilon} + C_R\right) - 2C \varepsilon .\]
By integrating the above inequality between $t=\varepsilon$ and $t$, we deduce thanks to \eqref{HJ:eq:J}
\[\begin{aligned}
 &  J_\varepsilon (t) > \left[J_\varepsilon(\varepsilon) +  \frac{2C \varepsilon^2 }{C_N}  + C_R \varepsilon \right] e^\frac{-C_N t }{\varepsilon} -\left(\frac{ 2C\varepsilon^2 }{C_N} + C_R \varepsilon \right) e^{C_N} \\
 &\geq \left[\frac{C_R C_N }{\varepsilon} +  \frac{2C \varepsilon^2 }{C_N}  + C_R \varepsilon \right] e^\frac{-C_N t }{\varepsilon} -\left(\frac{ 2C\varepsilon^2 }{C_N} + C_R \varepsilon \right) e^{C_N} \\
 &>-  O_\varepsilon(1).
\end{aligned} \]
Finally, following the Annex B of \cite{BMP}, we fix $\tau>0$ and it follows for $\varepsilon<\tau$
\[ \int_{\tau}^T | N_\varepsilon'(s)| ds = \int_{\tau}^T N_\varepsilon'(s)ds + 2\int_{\tau}^T \max( 0,  N_\varepsilon'(s)) ds \leq C_N - c_N  + 2(T - \tau) O_\varepsilon(1). \]
We conclude thanks to the compact embedding of $W^{1,1}([\tau,T])$ into $L^q([\tau, T])$. Up to a subsequence, $N_\varepsilon$ converges to a function $N$ on every interval of the form $[\tau, T]$ for every $\tau>0$. By a diagonal process, we conclude to the convergence of $N_\varepsilon$ on $]0, +\infty[$. Moreover, it is clear that $N$ is non-decreasing.


\vspace{0.5cm}

\textbf{$\bullet$ Convergence of $u_i^\varepsilon$. } From the points 1, 3 and 4 of Theorem\ref{HJ:thm:interm_u}, we deduce thanks to the Arzela-Ascoli Theorem that $u_i^\varepsilon$ converges uniformly on any set of the form $]0, R[ \times ]\tau,  T[$ with $R, T$ arbitrary large constants and $\tau$ an arbitrary small constant. We deduce that $u_i^\varepsilon$ converges uniformly locally on $[0, +\infty[ \times ]0, +\infty[$. Moreover, thanks to the point 2 of Theorem\ref{HJ:thm:interm_u}, we deduce that \[\underset{ \varepsilon \to 0}{\lim } \ u_1^\varepsilon (x,t)= \underset{ \varepsilon \to 0}{\lim } \ u_2^\varepsilon  (x,t)= u(x,t).\]
Next, we claim that $u(x,t) \leq 0$. We prove it by contradiction: assume that there exists a time $t>0$ and $x\in \mathbb{R}$ such that $u(x,t)>\alpha>0$. We deduce the existence of a sequence $\left((x_k, t_k), \varepsilon_k\right) \to \left((x,t), 0 \right)$ such that $u_i^{\varepsilon_k} (x_k, t_k)>\frac{\alpha}{2}$. Next, according to the point 3 of Theorem\ref{HJ:thm:interm_u}, there exists a radius $r>0$ such that for all $y \in B(x_k,r)$, there holds
\[ u_i^{\varepsilon_k}(y, t_k) > \frac{\alpha}{4}.\]
It follows that for $\varepsilon$ small enough, $N_\varepsilon > C_N$ which is in contradiction with the conclusion of Theorem\ref{HJ:thm:existence}. \\
We finally claim that for all $t>0$, we have $\underset{x \in \mathbb{R}_+ }{\sup}\  u(x,t) = 0$. Assume that the conclusion does not hold true. It follows the existence of a time $t>0$ such that $u(x,t) < -\alpha < 0$. We deduce that for $\varepsilon$ small enough, we have 
\[ u_i^\varepsilon (x, t) \leq \frac{-\alpha}{2} \qquad \forall x > 0.\]
We conclude that for $\varepsilon$ small enough, $N_\varepsilon < c_N$ which is in contradiction with the conclusion of Theorem\ref{HJ:thm:existence}. 

\vspace{0.3cm}

\textbf{$\bullet$ The function $u$ is solution of \eqref{HJ:eq:main_eq}. } We first prove that $u$ is is a super-solution in a viscosity sense of $\partial_t u - \mathcal{H}_D(\partial_x u, N) = 0$. We proceed as it was introduced in the article \cite{BMP}. Let $(x_0, t_0) \in \mathbb{R}^+ \times \mathbb{R}^+$ and $\phi$ be a regular test function such that 
\[ \min (u - \phi ) = (u - \phi)(x_0, t_0). \]
Then, we notice that 
\[u(x,t) =\underset{\varepsilon \to 0}{\lim} \  u_1^\varepsilon(y,s) - \varepsilon \ln\left(\psi_1^{\rho_0}\right)) = \underset{\varepsilon \to 0}{\lim} \  u_2^\varepsilon(y,s) - \varepsilon \ln\left(\psi_2\right) \]
where 
\[ \rho_0 =  \partial_x \phi (x_0, t_0) \quad \text{ and } \quad \psi^{\rho_0} = \begin{pmatrix} \psi_1^{\rho_0}  \\ \psi_2^{\rho_0}  \end{pmatrix} =   \begin{pmatrix}1 \\ \frac{ (d_1 - d_2) {\rho_0}^2 +  (r_1 - r_2)  +  \sqrt{(d_1 - d_2) {\rho_0}^2 + (r_1 - r_2)^2 + 4 \delta_1 \delta_2 }}{2\delta_2} \end{pmatrix}. \]
The function $\psi^{\rho_0}$, introduced in \eqref{HJ:def:psi}, is a positive eigenvector of ${\rho_0}^2 \textbf{D} + \textbf{R}$ associated to the eigenvalue $\mathcal{H}_D(\partial_x \phi (x_0, t_0), N_\varepsilon)$. We deduce 
\begin{equation}\label{HJ:eq:pfHJ1}
\begin{aligned}
 &\exists \varepsilon_k \underset{ k \to +\infty}{\longrightarrow} 0, \quad \varepsilon_k>0, \quad   \exists (x_k, t_k) \in \mathbb{R}^+ \times \mathbb{R}^+\qquad \text{ such that } (x_k ,t_k) \to (x_0, t_0) \\
 & (u_i^{\varepsilon_k}- \varepsilon_k \ln (\psi_i^{\rho_0}) -\phi)(x_k, t_k) = \min\left[(u_1^{\varepsilon_k}-  \phi) (x_k, t_k) , ( u_2^{\varepsilon_k} - \varepsilon_k \ln (\psi_2 ^{\rho_0}) - \phi)(x_k, t_k) \right]  , \\
 &\text{and }  \ [u_i^{\varepsilon_k} - \varepsilon_k \ln (\psi_i ^{\rho_0}) - \phi](x_k,t_k)  = \underset{\mathbb{R}^+ \times \mathbb{R}^+ }{\min}  [u_i^{\varepsilon_k} - \varepsilon_k \ln (\psi_i ^{\rho_0}) - \phi].
\end{aligned}
\end{equation}
As we have denoted $ \rho_0 =  \partial_x \phi (x_0, t_0) $, we will denote  $\rho_k =  \partial_x \phi (x_k, t_k) $. Notice that $\rho_k \to \rho_0$. Since it is a minimum point, it follows 
\[\begin{aligned}
&\partial_t  [u_i^{\varepsilon_k} - \varepsilon_k \ln (\psi_i ^{\rho_0}) - \phi](x_k,t_k)  = 0, \quad  \partial_x [u_i^{\varepsilon_k} - \varepsilon_k \ln (\psi_i ^{\rho_0} ) - \phi](x_k,t_k) =0 \quad\\
\text{ and } \quad -&\partial_{xx} [u_i^{\varepsilon_k} - \varepsilon_k  \ln (\psi_i ^{\rho_0}) - \phi](x_k,t_k) \leq 0.\end{aligned} \]
Using the equation \eqref{HJ:eq:main:HC}, we deduce that 
\[ 
0 \leq \left(  \partial_t \phi + \varepsilon_k d_i \partial_{xx}( \phi + \varepsilon_k \ln(\psi_i ^{\rho_0} )) - d_i [\partial_x \phi + \varepsilon_k \partial_{x}\ln(\psi_i ^{\rho_0} )]^2   - r_i +N_\varepsilon(t) -\delta_i e^{\frac{u_j^{\varepsilon_k} - u_i^{\varepsilon_k}}{\varepsilon_k}} \right) (x_k,t_k) .
\]
Moreover, according \eqref{HJ:eq:pfHJ1},  we have 
\[ (u_i^\varepsilon  - u_j^\varepsilon )(x_k, t_k ) \leq  \varepsilon_k \left[\ln(\psi_i^{\rho_0}(x_k) ) - \ln( \psi_j ^{\rho_0}(x_k) )\right].\]
It follows 
\[\begin{aligned} 0 \leq &\left( \partial_t \phi + \varepsilon_k d_i \partial_{xx}(\phi+\varepsilon_k \ln(\psi_i ^{\rho_0})) - d_i (\partial_x \phi   + \varepsilon_k  \partial_x \ln(\psi_i ^{\rho_0}))^2  - r_i+N_\varepsilon(t) - \frac{\delta_i \psi_j ^{\rho_0}}{\psi_i ^{\rho_0}} \right) (x_k,t_k) \\ 
= & \partial_t \phi(x_k, t_k) - \left( d_i (\partial_x \phi)^2 -r_i + N_\varepsilon - \frac{\delta_i \psi_j^{\rho_k}}{\psi_i^{\rho_k}}   \right) (x_k,t_k) \\
+&  d_i\left( \varepsilon_k \partial_{xx} (\phi+\varepsilon_k \ln(\psi_i ^{\rho_0})) - 2\varepsilon_k \partial_x \phi \partial_x \ln(\psi_i ^{\rho_0}) - \varepsilon_k^2  ( \partial_{x} \ln(\psi_i ^{\rho_0}))^2 \right) (x_k, t_k) - \delta_i(x_k) \left(\frac{ \psi_j^{\rho_k}}{\psi_i^{\rho_k}} -\frac{ \psi_j^{\rho_0}}{\psi_i^{\rho_0}} \right) \\
& =  \partial_t \phi(x_k, t_k) - \left( d_i (\partial_x \phi)^2  + r_i - N_\varepsilon + \frac{\delta_i \psi_j^{\rho_k}}{\psi_i^{\rho_k}}   \right) (x_k,t_k) + o_{\varepsilon_k}(1).
\end{aligned}\]
Moreover, recalling that $\mathcal{H}_D(\partial_x \phi(x_k, t_k), N_\varepsilon)$ is an eigenvalue of $\partial_x \phi^2(x_k, t_k) \textbf{D} + \textbf{R}$, it follows
\[\begin{aligned}
 & (d_i[\partial_x \phi]^2 + r_i - N_{\varepsilon_k}) \psi_i^{\rho_k} +  \left(\delta_i \psi_j^{\rho_k} \right)(x_k, t_k) = \left(\mathcal{H}_D(\partial_x \phi, N_\varepsilon ) \psi_i^{\rho_k} \right)(x_k, t_k) \\
 \Rightarrow \quad &\left( d_i (\partial_x \phi)^2 -r_i + N_\varepsilon - \frac{\delta_i \psi_j^{\rho_k}}{\psi_i^{\rho_k}}   \right)(x_k, t_k)  = \mathcal{H}_D(\partial_x \phi, N_\varepsilon) (x_k, t_k).
\end{aligned}\]
We deduce that 
\[ 0 \leq \partial_t \phi (x_k, t_k)  - \mathcal{H}_D(\partial_x \phi, N_{\varepsilon_k} )(x_k , t_k) +o_{\varepsilon_k}(1) .\]
Taking the limit $k \to +\infty$, we conclude that $u$ is a super-solution of $\partial_t u - \mathcal{H}_D(\partial_x u , N) = 0$ in a viscosity sense.

\vspace{0.5cm}

It remains to prove the limit conditions: we verify that $u$ satisfies in a viscosity sense $-\partial_x u (x=0, t) = 0$. Let $\phi$ be such that $u-\phi$ takes its minimum at $x=0$ and for some positive time $t$. We deduce the existence of $(x_\varepsilon, t_\varepsilon)$ such that 
\[\begin{aligned}
    &x_\varepsilon \to 0, \quad t_\varepsilon \to t \text{ as } \varepsilon \to 0,  \\
&[u_i^{\varepsilon}- \varepsilon \ln (\psi_i [\partial_x \phi (0, t)])  - \phi](x_\varepsilon, t_\varepsilon )= \min\left((u_1^{\varepsilon} - \phi) (x_\varepsilon, t_\varepsilon),  (u_2^{\varepsilon} - \varepsilon \ln (\psi_2[\partial_x \phi (0, t)])) (x_\varepsilon, t_\varepsilon) \right)\\
    \text{ and } \quad & [u_i^{\varepsilon}- \varepsilon \ln (\psi_i[\partial_x \phi (0, t)])  - \phi](x_\varepsilon, t_\varepsilon ) = \underset{(x,t) \in \mathbb{R}^+ \times \mathbb{R}^+ }{\min } [u_i^{\varepsilon}- \varepsilon \ln (\psi_i[\partial_x \phi (0, t)])  - \phi](x,t).
\end{aligned}   \]
We distinguish two cases: 
\begin{enumerate}
    \item \textit{Case 1 : $x_\varepsilon >0$.} In this case, we conclude exactly as above that 
    \[ 0 \leq \partial_t \phi (x_\varepsilon, t_\varepsilon)  - \mathcal{H}_D(\partial_x \phi(x_\varepsilon, t_\varepsilon), N_\varepsilon ) +o_{\varepsilon}(1) .\]
    \item \textit{Case 2 : $x_\varepsilon = 0$.} In this case, using the fact that $(x_\varepsilon, t_\varepsilon)$ is a minimum point, we deduce that 
    \[ -\partial_x [u_i^{\varepsilon}- \varepsilon \ln (\psi_i[\partial_x \phi (0, t)])  - \phi](0, t_\varepsilon )\leq 0.\]
    Next, according to the Neumann boundary conditions imposed to $u_i^\varepsilon$, we deduce that 
    \[ \varepsilon \partial_x( \ln (\psi_i [\partial_x \phi (0, t)]) ) (0) \leq -\partial_x \phi(0, t_\varepsilon) . \]
\end{enumerate}
Passing to the superior limit $\varepsilon \to 0$
\[ 0 \leq  \max\left( -\partial_x \phi(0, t) , \phi (0, t)  - \mathcal{H}_D(\partial_x \phi(0, t) , N_\varepsilon (t))\right)\]
which corresponds to the boundary conditions in a viscosity sense. 

\vspace{0.5cm}

The proof that $u$ is a sub-solution of \eqref{HJ:eq:main_eq} follows from the same arguments.

\vspace{0.5cm}

\textbf{$\bullet$ Convergence of $n_i^\varepsilon$ in the sense of measures. } The proof that $n_i^\varepsilon$ converges to a measure follows from the convergence of $u_i^\varepsilon$ towards $u$. Indeed, fix times $0 < t_1< T$; then, according to point 1. of Theorem\ref{HJ:thm:interm_u}, there exists $R_T>0$ such that for any $x>R_T$, $t\in (t_1, T)$ and $\varepsilon$ small enough, we have 
\[ u_i^\varepsilon (x,t) \leq C(1 -  x) \quad \text{ and } \quad   \int_{\left\lbrace x > R_T \right\rbrace } n_1^\varepsilon(x,t) + n^\varepsilon_2(x,t) dx \leq \frac{c_N}{2}. \]
Hence, we deduce that $n_i^\varepsilon \to 0$ on $\left\lbrace x > R_T \right\rbrace$. It follows 
\[ \begin{aligned}
 \frac{ c_N }{2\left(1 + \underset{ x \in [0, R_T]}{\min} q_1(x)e^{\frac{\varepsilon^3}{t_1}} \right) }&\leq  \frac{ 1 }{\left(1 + \underset{ x \in [0, R_T]}{\min} q_1(x)e^{\frac{\varepsilon^3}{t_1}} \right) }  \int_{\left\lbrace x < R_T \right\rbrace } (n_1^\varepsilon (x,t) + n_2^\varepsilon (x,t)) dx  \\
 &\leq   \frac{ 1 }{\left(1 + \underset{ x \in [0, R_T]}{\min} q_1(x)e^{\frac{\varepsilon^3}{t_1}}\right) } \int_{\left\lbrace x < R_T \right\rbrace } n_1^\varepsilon (x, t) (1 + q_1(x) e^{\frac{\varepsilon^3}{t}}) dx  \\
 &\leq   \int_{\left\lbrace x < R_T \right\rbrace } n_1^\varepsilon (x, t) dx  \leq C_N.
\end{aligned}  \]
Since the same type of inequalities is valid for $n_2^\varepsilon$, we deduce that up to an extraction we have $(n_1^\varepsilon, n_2^\varepsilon) \to (n_1, n_2)$, where $n_1$ and $n_2$ are two non-trivial measures. Next, we prove that 
\[ \mathrm{supp} \  n_i(\cdot, t) \subset \left\lbrace u(\cdot, t) = 0 \right\rbrace. \]
Let a time $t>0$ and $\phi$ be a positive regular compactly supported test function, such that 
\[ \mathrm{supp} \  \phi \subset \left\lbrace u(\cdot, t) = 0 \right\rbrace^c.\]
We deduce that there exists $a>0$ such that 
\[ \underset{ x \in \mathrm{supp} \ \phi}{\max} u( x , t ) < -a. \]
Hence, for $\varepsilon$ small enough, we have $\underset{ x \in \mathrm{supp} \ \phi}{\max} u_i^\varepsilon( x , t ) < -\frac{a}{2}$. The conclusion follows the following computation:
\[ \int_{\mathbb{R}^+} \phi(x) n (x, t) dx = \underset{\varepsilon \to 0}{\lim} \int_{\mathbb{R}^+} \phi(x) n_i^\varepsilon (x, t) dx  \leq  \underset{\varepsilon \to 0}{\lim} \int_{\mathbb{R}^+} \phi(x) e^{-\frac{a}{2\varepsilon}} dx  = 0. \]
\flushright\qed
\end{proof}

\section{An application: optimal timing in the adaptation to DNA damage}\label{sec5:Motiv}

\subsection{A general non-local system modelling adaptation to DNA damage}
When eukaryotic cells face damage to their DNA, specialised mechanisms come into play. The DNA damage checkpoint signalling pathway leads to stopping the cell cycle at the G2/M phase. Then appropriate repair pathways are activated. These mechanisms are called the DNA damage response.

However, if the repair pathways fail for too long, then the cells will override the DNA damage checkpoint and resume cell division even though the damage is still present \cite{TGH,LMHUKH}; they will do so at a variable timing, ranging from 5 to 15 hours in the budding yeast model organism \cite{SZ}. This phenomenon is called adaptation to DNA damage. Due to improper chromosome segregation \cite{KMCVHT}, adapted cells have chromosomal instability and a high mortality rate, making adaptation a last resort mechanism after all repair options have already failed. This leads to a hierarchy of cell fate decisions: repair is attempted first and then the cells adapt.

Although the underlying molecular mechanisms of adaptation are not yet completely understood, the consequences of the timing and heterogeneity of adaptation in the survival of the population were investigated through a mathematical model and numerical simulations in \cite{RSX}. The authors propose to model the population by a system of ordinary differential equations:
\begin{equation*}
	\dfrac{d }{d t} \left(    \begin{array}{c}  D(t)\\ A(t)\\ R(t)   \end{array}      \right) =  \left(    \begin{array}{ccc}  -\gamma_d - \beta(x,p,t) -\alpha(t) & 0 & 0 \\ \beta(x,p,t) & -\gamma_{a} - \delta & 0 \\ \alpha(t) & \delta & 0  \end{array}      \right)\left(    \begin{array}{c}  D(t)\\ A(t)\\ R(t)   \end{array}      \right) +\  \left(    \begin{array}{c}  0\\ A(t)(1-\frac{N(t)}{N_{max}})\\ R(t)(1-\frac{N(t)}{N_{max}})   \end{array}      \right),
\end{equation*}
where, at time $t$, $D(t)$ is the quantity of damaged cells, $A(t)$ the quantity of adapted cells and $R(t)$ the quantity of healthy cells whose DNA is repaired. Last, we denote
\[N(t) = A(t)+R(t)+D(t)\]
the total population.

The initial population is composed of a quantity $D(0)$ of damaged cells. They repair their damaged DNA with a rate 
\begin{equation}
	\alpha(t) =  \alpha_m \e^{-\frac{(t-\mu_a)^2}{2\sigma}},
\end{equation}
with  $\alpha_m,\sigma,\mu_a\in\R_+^*$, and they adapt with a rate
\begin{equation}
	\beta( x, p , t) = \dfrac{\beta_m}{1+ \e^{-p(t- x )}},
\end{equation}
with $\beta_m,p\in\R_+^*$. The adapted cells have access to other repair mechanisms at later stages of the cell cycle and we assume that they manage to repair their DNA damage at rate $\delta\in\R_+^*$ after adaptation. The values $\gamma_d$ and $\gamma_a$ are the death rates of damaged and adapted cells respectively; the death rate of healthy cells is assumed to be $0$ for the sake of clearness.

Depending on the value of $x\in\R_+$, which represents the timing of the adaptation process, the population will take a certain time $T_S(x)$ to reach some arbitrary level near the carrying capacity $N_{max}$ of the system. The authors of \cite{RSX} observe that there exists for most parameters an optimal value $x^*$ which minimises $T_S$, thus allowing the population to grow back to a healthy size as fast as possible after an external event has damaged the DNA of all cells. The authors also investigate the dependency with respect to the parameter $p$ which represents the heterogeneity of the adaptation timing and they provide arguments for the hypothesis that an optimal value for $p$ can be selected by a bet-hedging mechanism.\\

Here we go further into investigating the selection of an optimal adaptation timing $x^*$. Instead of studying for each value $x$ of the genetic trait the evolution of the population, we consider a population of cells with varying genetic trait competing for the same resources. We also add genetic diffusion  for healthy and adapted cells and a continuous source of damage $\mathcal D(t)$.

Let $n( x ,t)$ represent at time $t$ the density of healthy cells with genetic trait $ x $; let $d( x ,s,t)$ represent at time $t$ the density of cells with genetic trait $ x $ whose DNA is damaged since a time $s$; let $a( x ,t)$ represent the density of adapted cells at time $t$ with genetic trait $ x $. We also introduce a scaling parameter $\varepsilon\in\R_+^*$.

The repair probability can now take into account both an absolute time part and a "time since the damage occured" part:
\begin{equation}
	\alpha(s,t) =  \bar \alpha(t) \e^{-\frac{(s-\mu_a)^2}{2\sigma}},
\end{equation}
where the function $\bar\alpha:\R_+\to [0,\alpha_m]$ allows us to take into account environmental events that prevent cells from repairing their DNA damage. The adaptation probability can depend upon $x$ or $p$, which we will denote for clarity
\begin{equation}
	\beta( x, s) = \dfrac{\beta_m}{1+ \e^{-p(s- x )}}, \qquad or \qquad \beta( p , s) = \dfrac{\beta_m}{1+ \e^{-p(s- x )}},
\end{equation}
to indicate if cells vary along genetic trait $x$ or $p$ in the model.

The model writes
\begin{equation}\label{eqn:horrible_1}
	\varepsilon\dfrac{\partial n}{\partial t}( x ,t) - \varepsilon^2d_1\dfrac{\partial^2 n}{{\partial  x }^2} ( x ,t) = n( x ,t)\big( 1 - \mathcal D(t) - N(t)  \big)    +\delta a( x ,t) + \int_{0}^{+\infty} \alpha(s)d( x ,s,t) \diff s,
\end{equation}
\begin{equation}\label{eqn:horrible_2}
	\varepsilon\dfrac{\partial d}{\partial t}( x , s, t) + \dfrac{\partial d}{\partial s}( x , s, t) + \big(\gamma_d + \alpha(s,t) + \beta( x ,s) \big)d( x ,s,t) = 0,
\end{equation}
\begin{equation}\label{eqn:horrible_3}
	\varepsilon\dfrac{\partial a}{\partial t}( x ,t) - \varepsilon^2d_2\dfrac{\partial^2 a}{{\partial  x }^2}( x ,t)=  a( x ,t)(1-\gamma_a - \delta - N(t) )   + \int_0^{+\infty}\beta( x ,s) d( x ,s,t) \diff s,
\end{equation}
\begin{equation}\label{eqn:horrible_4}
	N(t) = \int_{0}^{+\infty} \left( n( x ,t) + \int_0^{+\infty} d( x ,s,t)\diff s  + a( x ,t)\right)  \diff  x ,
\end{equation}
with the following initial and boundary conditions
\begin{equation}
	\left\{ \begin{array}{l}
	\displaystyle \dfrac{\partial n}{\partial x}(0,t)= \dfrac{\partial a}{\partial x}(0,t) = 0,\qquad t\in\R_+, \\
	\displaystyle d( x ,0,t) = \mathcal D(t) n( x ,t),\qquad  x ,t\in\R_+,\\
	\displaystyle n( x ,0)=n^0( x ), \quad d( x ,s,0)= d^0( x ,s), \quad a( x ,0) = a^0( x ), \qquad  x \in\R_+,
	\end{array}\right.
\end{equation}
and the constants $d_1, d_2, \delta, \gamma_d, \gamma_a \in\R_+^*$.\\

\subsection{Simplification into a two populations system}\label{sec:3.2}

This system of non-local partial differential equations is complicated and very hard to tackle numerically. Hence, we simplify the dynamics of the damaged cells by making the quasi-static approximation
\[
\partial_s d(x,s,t)  +(\gamma_d+ \alpha(s,t) +\beta(x,s)) d(x,s,t) =0.
\]
Then, we can compute the quantity of damaged cells explicitely:
\[
d(x,s,t) = \mathcal D(t) n_\varepsilon(x,t) e^{- \gamma_d s - \int_0^s \alpha(z,t)dz - \int_0^s \beta(x,z)dz}. 
\]
We also make the simplifying assumption that the damage rate is constant, \textit{i.e.} $\mathcal D(t) = D > 0$, and the new total mass is given by
\[N_\varepsilon(t) = \int_0^{+\infty} (n_\varepsilon(x,t)+a_\varepsilon(x,t)) dx.\]
Hence, we come to the simplified model
\begin{equation}\label{eqn:reduced_model}
\left\{
\begin{array}{rcl}
     \varepsilon \partial_t n_\varepsilon(x,t) - \varepsilon^2 d_1 \Delta  n_\varepsilon(x,t) & = &\displaystyle n_\varepsilon(x,t) (1-D-N_\varepsilon(t)) + \delta a(x,t)\\ &&\displaystyle \qquad  +Dn_\varepsilon(x,t)\int_0^\infty \alpha(s,t) e^{- \gamma_d s - \int_0^s \alpha(z,t)dz - \int_0^s \beta(x,z)dz}ds,\\
     \varepsilon \partial_t a_\varepsilon(x,t) - \varepsilon^2  d_2 \Delta a_\varepsilon(x,t) & = &\displaystyle  a_\varepsilon(x,t) (1-\gamma_a- \delta -N_\varepsilon(t))\\&&\displaystyle \qquad +Dn_\varepsilon(x,t)\int_0^\infty \beta(x,s) e^{- \gamma_d s - \int_0^s \alpha(z,t)dz - \int_0^s \beta(x,z)dz}ds,\\
     N_\varepsilon(t) & = &\displaystyle \int_0^{+\infty} (n_\varepsilon(x,t)+a_\varepsilon(x,t)) dx,
\end{array}
\right.
\end{equation}
with the following initial and boundary conditions
\begin{equation}
	\left\{ \begin{array}{l}
	\displaystyle \dfrac{\partial n_\varepsilon}{\partial x}(0,t)= \dfrac{\partial a_\varepsilon}{\partial x}(0,t) = 0,\qquad t\in\R_+, \\
	\displaystyle n_\varepsilon( x ,0)=n^0( x ), \quad a_\varepsilon( x ,0) = a^0( x ), \qquad  x \in\R_+.
	\end{array}\right.
\end{equation}
If we choose $\bar \alpha(t) = \alpha_m$, \textit{i.e.} $\alpha(s,t) = \alpha(s)$, and if we denote
\begin{multline*}
    r_1(x) = 1-D + D\int_0^\infty \alpha(s) e^{- \gamma_d s - \int_0^s \alpha(z)dz - \int_0^s \beta(x,z)dz}ds, \qquad r_2(x) = 1-\gamma_a- \delta,\\  \delta_1(x) = \delta \qquad \text{ and } \qquad \delta_2(x) =  D \int_0^\infty \beta(s) e^{- \gamma_d s - \int_0^s \alpha(z)dz - \int_0^s \beta(x,z)dz} d s,
\end{multline*}
the sytem \eqref{eqn:reduced_model} is of the form \eqref{HJ:eq:main}.

If we assume that $D < 1$ and $\gamma_a < 1$, then the functions $r_1, r_2, \delta_1, \delta_2$ defined above satisfy assumptions \eqref{HJ:Hyp:Rdelta} and \eqref{HJ:Hyp:R}. Most of the conditions can be readily checked and we postpone the remaining technicalities to the Appendix \textcolor{red}{B}. For \eqref{HJ:Hyp:Rdelta} the only difficult part is to check that $e^{C_\delta x}\delta_2(x) \underset{x \to +\infty}{\longrightarrow} +\infty$, which is granted thanks to Lemma \ref{lm:C1}. For \eqref{HJ:Hyp:R}, thanks to Lemma \ref{lm:C2} and using the fact that
\[  \int_0^\infty \alpha(s) e^{- \gamma_d s - \int_0^s \alpha(z)dz - \int_0^s \beta(x,z)dz}ds \leqslant 1,   \]
we can choose
\[ c_N = \min(1-\delta_a,1-D)\qquad \mathrm{and} \qquad C_N = 2+D.\]
Therefore, we can apply Theorem\ref{HJ:thm:main_thm} and Theorem\ref{HJ:thm:interm_u}. In particular, if $d_1 = d_2$, then for all $\varepsilon\in\R_+^*$,
\[ \dfrac{n_\varepsilon(x,t)}{a_\varepsilon(x,t)} \underset{t\to+\infty}{\longrightarrow} q(x)  = \frac{ r_1(x) - r_2(x)  + \sqrt{ (r_1(x) - r_2(x))^2 + 4 \delta_1(x) \delta_2(x)}}{2 \delta_2(x)}. \]

\begin{figure}[h!]
	\begin{center}
		\includegraphics[width=300pt]{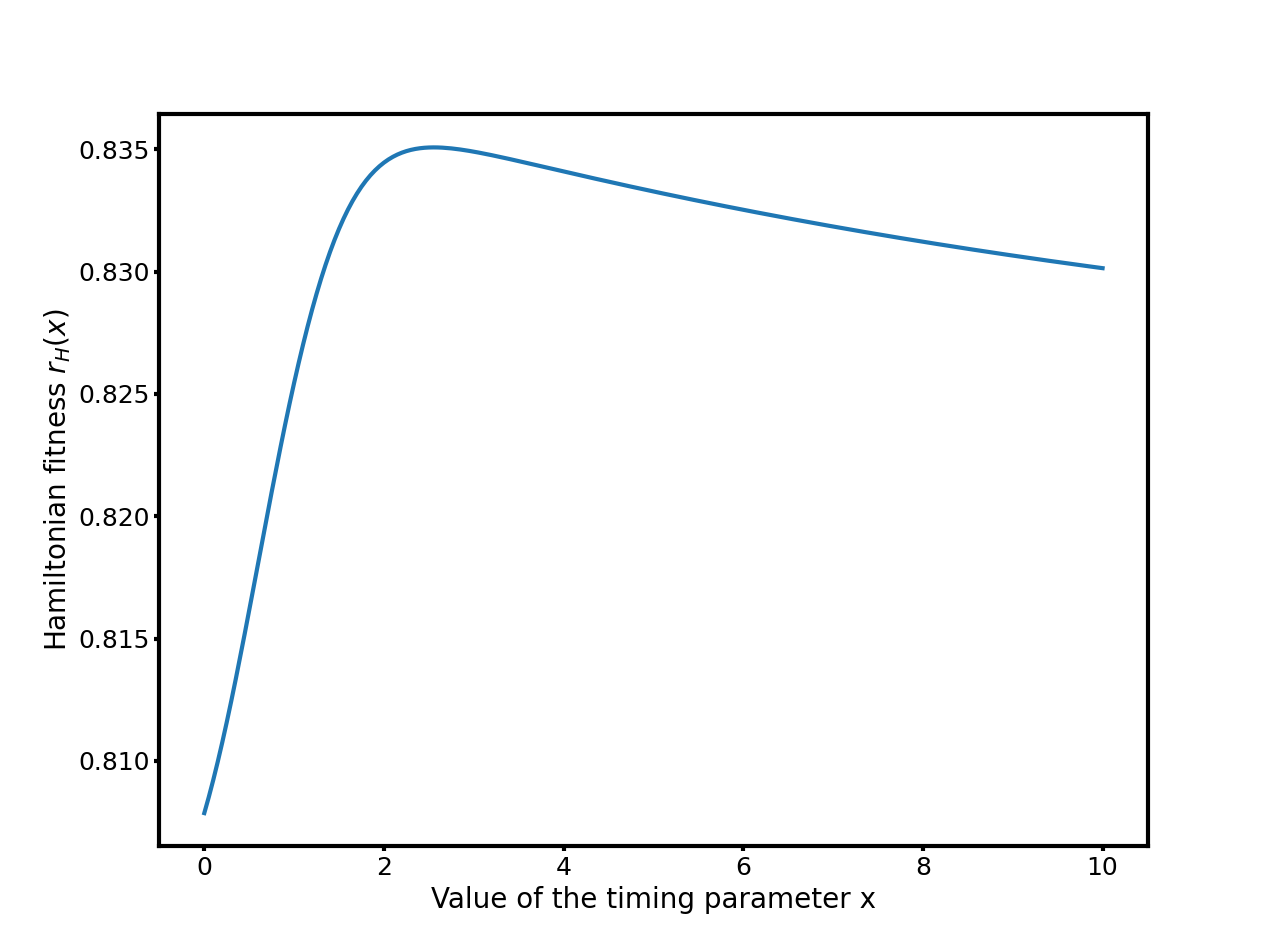}
	\end{center}
	\caption{Hamiltonian fitness $r_H(x)$ of the system for $p=3$. \label{fig:fitness_global}}	
\end{figure}

Moreover, recall that in the case $d_1=d_2=1$ the Hamiltonian defined in \eqref{HJ:def:H} can be decomposed into
\[ \mathcal H(\rho,N) = \rho^2 + r_H(x) - N, \]
with
\[ r_H(x) = \frac12\left(r_1 + r_2 + \sqrt{(r_1-r_2)^2 + 4 \delta_1 \delta_2 }\right),  \]
the \textit{Hamiltonian fitness}. This function also describes the stationary states as explained in Section \ref{HJ:sec:LongTermBehavior}. We can compute numerically this function $r_H$ to gain insights about the behaviour of the system in the limits $\varepsilon\to 0$ or $t\to+\infty$.

When the variable of interest is the mean time of adaptation $x$, with fixed $p=\bar p$, $r_H(x)$ has a unique global maximum as can be seen on Figure \ref{fig:fitness_global}. The numerical results in the next section indicate that when $\varepsilon$ goes to $0$, the solutions concentrate on a Dirac mass moving towards the maximum point. Hence, this model strengthens the hypothesis of \cite{RSX} that an optimal timing $x^*$ for adaptation tend to be favored by natural selection other long timescales. Here, this optimal time is expressed as
\[ x^* = \underset{x\in\R_+}{\mathrm{argmax}} \ r_H(x). \]
Let us mention that $r_H(x)$ should also drive the profile of the stationary state for fixed $\varepsilon$, since, as mentioned above, the formal equation for $w(x) = n_{1,\infty}(x)+n_{2,\infty}(x)$ is
\[ \varepsilon^2 \Delta w(x) - r_H(x)w(x) = 0.  \]

\begin{figure}[h!]
	\begin{center}
		\includegraphics[width=450pt]{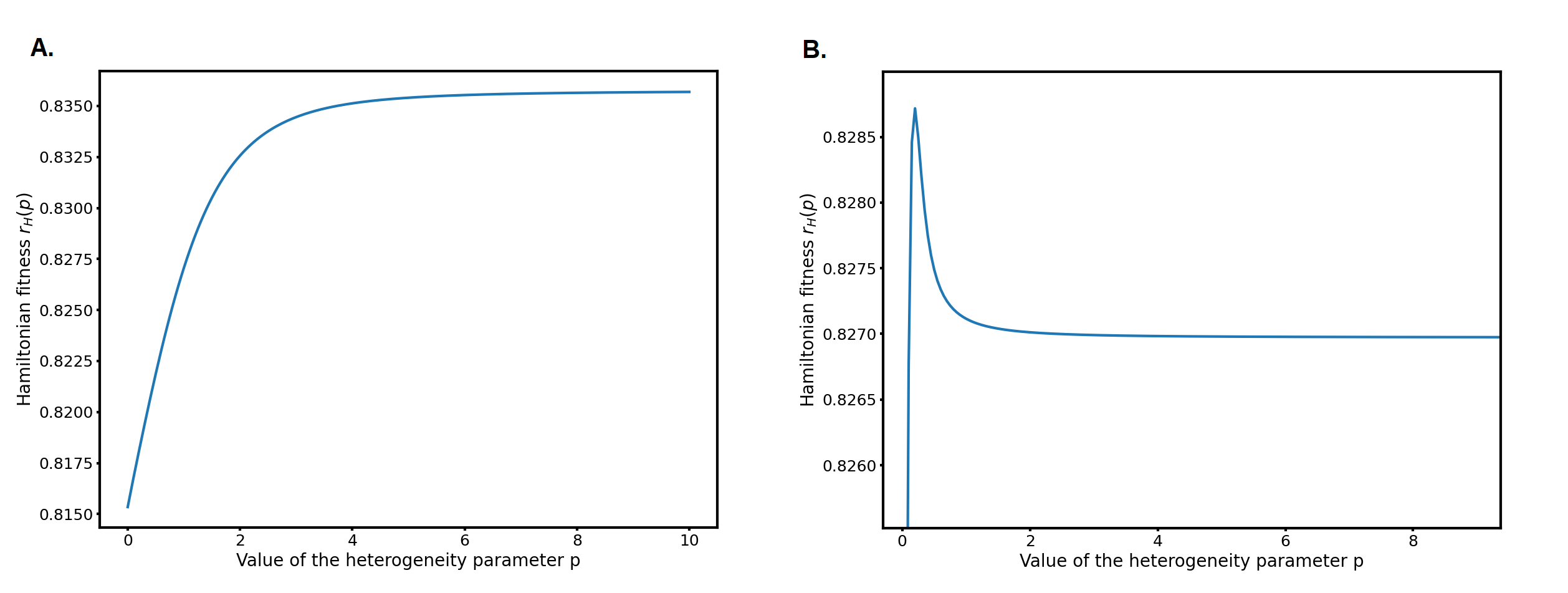}
	\end{center}
	\caption{Hamiltonian fitness $r_H(p)$ of the system for (A.) $\bar x=2$ and (B.) $\bar x=20$. \label{fig:fitness_global_p}}	
\end{figure}

If we fix an adaptation timing $x = \bar x$ and we take as a variable the adaptation heterogeneity parameter $p$, we can compute another equivalent fitness $r_H(p)$ which is displayed in Figure \ref{fig:fitness_global_p}.

As can be seen in Figure \ref{fig:fitness_global_p}A, for "reasonable" values of $\bar x$ the function $r_H:p\mapsto r_H(p)$ is increasing on $\R_+$. As we can observe in the numerical simulations in the following section, when $\varepsilon\to0$ the solutions concentrate on a Dirac mass that moves towards $+\infty$. This is in accordance with the findings of \cite{RSX} in the simpler ODE model: when the environment is predictable, the optimal strategy for the cells is to minimise the variance around any "good enough" adaptation timing, which amount to taking the largest possible value for $p$.

If the mean adaptation time $\bar x$ is large enough, for example $\bar x=20$, it can be seen in Figure \ref{fig:fitness_global_p}B that $r_H(p)$ has a unique global maximum. Since adaptation is really late, a smaller value $p^*$ (i.e. a larger variance for the adaptation) is selected to compensate.

However, as we said in the beginning of this section, in real life experiments the cells adapt with a variable timing. In \cite{RSX}, this fact was explained as a bet-hedging mechanisms in an unpredictable environment. When the optimisation procedure in the variable $p$ has to take into account a random variable in the repair function $\alpha$, a particular value $p^*$ is selected. Here we use the absolute-time part $\bar\alpha(t)$ in the repair function $\alpha(s,t)$ to model the changing environment. In the next section, we also make numerical experiments to explore what happens to the solution with a time-periodic $\bar \alpha$ function.

\section{Numerical simulations}\label{sec6:Num}

In this section, we investigate numerically the behaviour of the system \eqref{eqn:reduced_model} with the parameters and functions described in Section \ref{sec:3.2}.

We use a standard Cranck-Nicolson scheme for the Laplacians and the reaction terms are treated explicitly. The scheme uses an artificial Neumann boundary condition in the right end of the domain which has no impact over the results whatsoever as long as the numerical spatial domain is large enough and the scaling parameter $\varepsilon$ is small enough. This rather simple scheme appears to be very robust even for small $\varepsilon$ values, as long as the time step $dt$ is of the same order of magnitude of $\varepsilon$. The Python code we used to produce the numerical simulations is available at \url{https://github.com/pierreabelroux/Leculier_Roux_2021}. The figures can be obtained by uncommenting and running the different pieces of code in the part "Numerical experiments".

\begin{figure}[h!]
	\begin{center}
		\includegraphics[width=450pt]{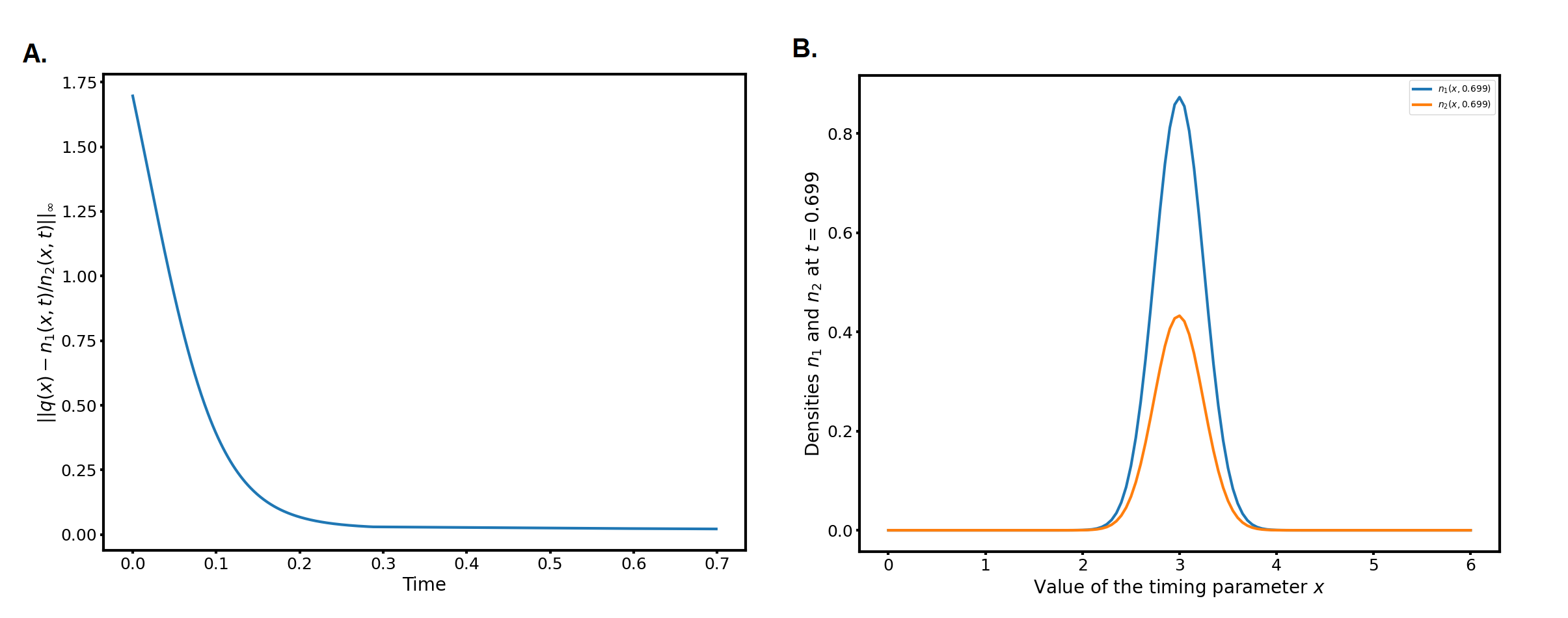}
	\end{center}
	\caption{ (A.) Distance between $n_1/n_2$ and $q$ in $L^\infty$ norm. The norm is computed over the finite numerical domain (B.) Plot of $n_1(\cdot,t)$ and $n_2(\cdot,t)$ for $t=0.699$. The diffusion parameters are $d_1=d_2=1$. \label{fig:cvToQ}}	
\end{figure}

As can be seen on Figure \ref{fig:cvToQ}, the convergence of the quantity $\norme{q(\cdot) - n_1(\cdot,t)/n_2(\cdot,t)}$ is very fast and the shapes of $n_1$ and $n_2$ are similar right after a short transitory period. In this figure we take 
\[n_1^0 = n_2^0 = \frac15 e^{-10(x-3)^2},\]
to avoid visual scaling problems with $n_1/n_2$ (this quotient can be very large in the first milliseconds for Gaussians with distant means) but it does not affect the speed of convergence towards $q(x)$ which is consistent across all types of initial data.

Consequently, we will only plot $n_1$ in the following numerical experiments for the sake of clarity. 

\subsection{Evolution along the parameter $x$ for fixed $p$}

\begin{figure}[h!]
	\begin{center}
		\includegraphics[width=450pt]{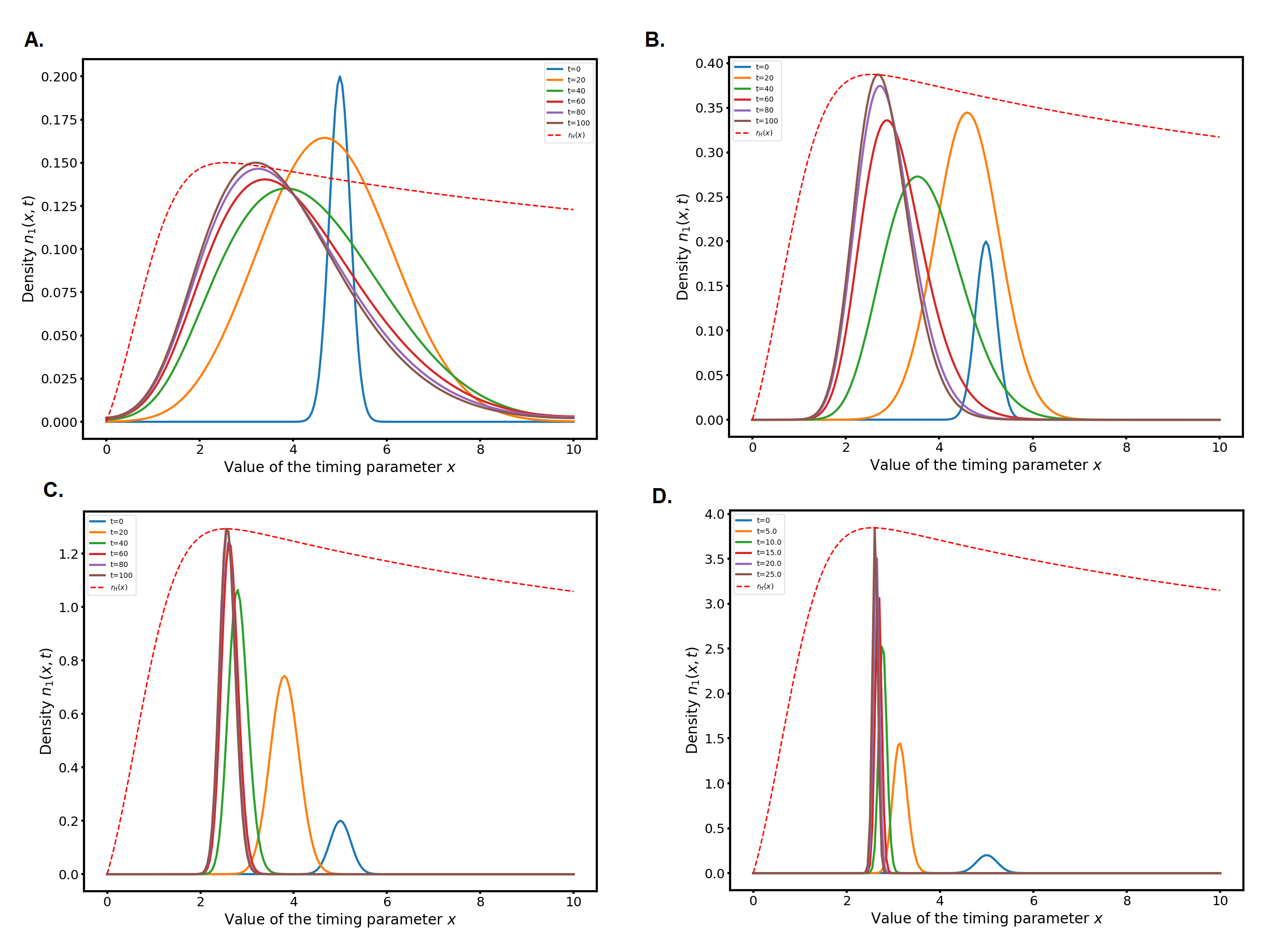}
	\end{center}
	\caption{Evolution in time of $n_1(x,t)$ from the same initial data for $p=3$ and different values of $\varepsilon$. The dashed line is the Hamiltonian fitness $r_H(x)$ which is re-scaled for the sake of readability. (A.) $\varepsilon = 0.05$ (B.) $\varepsilon = 0.01$ (C.) $\varepsilon = 0.001$ (D.) $\varepsilon = 0.0001$  \label{fig:evolution_x}}	
\end{figure}

For the fixed values $p=3$ and $d_1=d_2=1$, we simulate the system \eqref{eqn:reduced_model} for different values of $\varepsilon$ (see Figure \ref{fig:evolution_x}). As predicted by our theoretical results, when $\varepsilon$ tends to 0 the solution behaves like a Dirac mass moving towards the maximum point of the Hamiltonian fitness $r_H(x)$. This corresponds to the selection of an optimal mean value for the timing of the adaptation process, which provides an evolutionary explanation for the results of the laboratory experiments on budding yeasts \cite{CX}.

Yet, taking $d_1=d_2$ is not realistic from a biological point of view in this context. It is observed in experiments that adapted cells have a more unstable genome and thus the genetic diffusion might be very asymmetric \cite{CXMLMPCDT,CX}. Our theoretical setting gives us less clear results in the case $d_1\neq d_2$ because we can't define and simulate in a simple way a Hamiltonian fitness $r_H(x)$ to see were are the optimal traits: the Hamiltonian rather decompose into
\[ \mathcal H(\rho,N) = \dfrac{d_1+d_2}{2}\rho^2 + r_H(x,\rho) - N,\]
and the function $r_H$ then involves the gradient of the solution, which is evolving in space and time.

\begin{figure}[h!]
	\begin{center}
		\includegraphics[width=450pt]{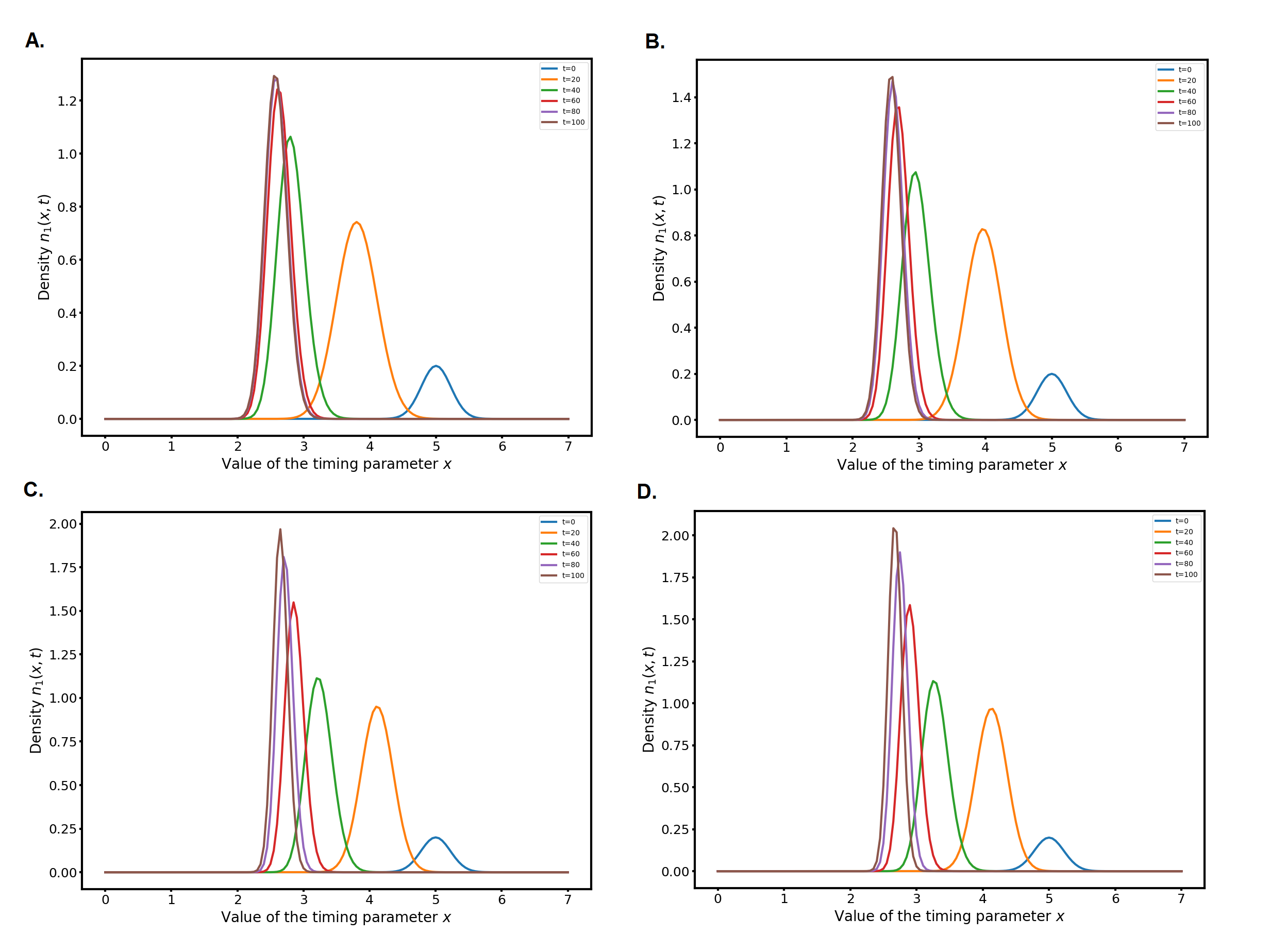}
	\end{center}
	\caption{Evolution in time of $n_1(x,t)$ from the same initial data for $p=3$, $\varepsilon=0.001$ and different pairs $(d_1,d_2)$ with same sum $d_1+d_2$ (A.) $d_1 = 1$, $d_2=1$ (B.)  $d_1 = 0.5$, $d_2=1.5$ (C.)  $d_1 = 0.05$, $d_2=1.95$ (D.)  $d_1 = 0$, $d_2=2$. \label{fig:evolution_d}}	
\end{figure}

Therefore, we run numerical experiments for $\varepsilon=0.001$ and different values of $d_1$ and $d_2$ (see Figure \ref{fig:evolution_d}) to see how it impacts the evolution of the solutions in time. It appears that the overall behaviour of the system is not changed much by the different values. The higher diffusion $d_2$ drives the evolution and even in the extreme case $d_1=0$ the qualitative behaviour is slower but similar to the case $d_1=d_2=1$. This last case, when the genetic diffusion is assumed to be negligible in healthy cells, is of particular interest for biologists for it allows to investigate adaptation to DNA damage as a mechanism promoting genetic diversity of organisms \cite{CX}. The stability of the model with respect to this particular case strengthens this hypothesis.

\subsection{Evolution along the parameter $p$ for fixed $x$}

\subsubsection{Stable environment}

\begin{figure}[h!]
	\begin{center}
		\includegraphics[width=450pt]{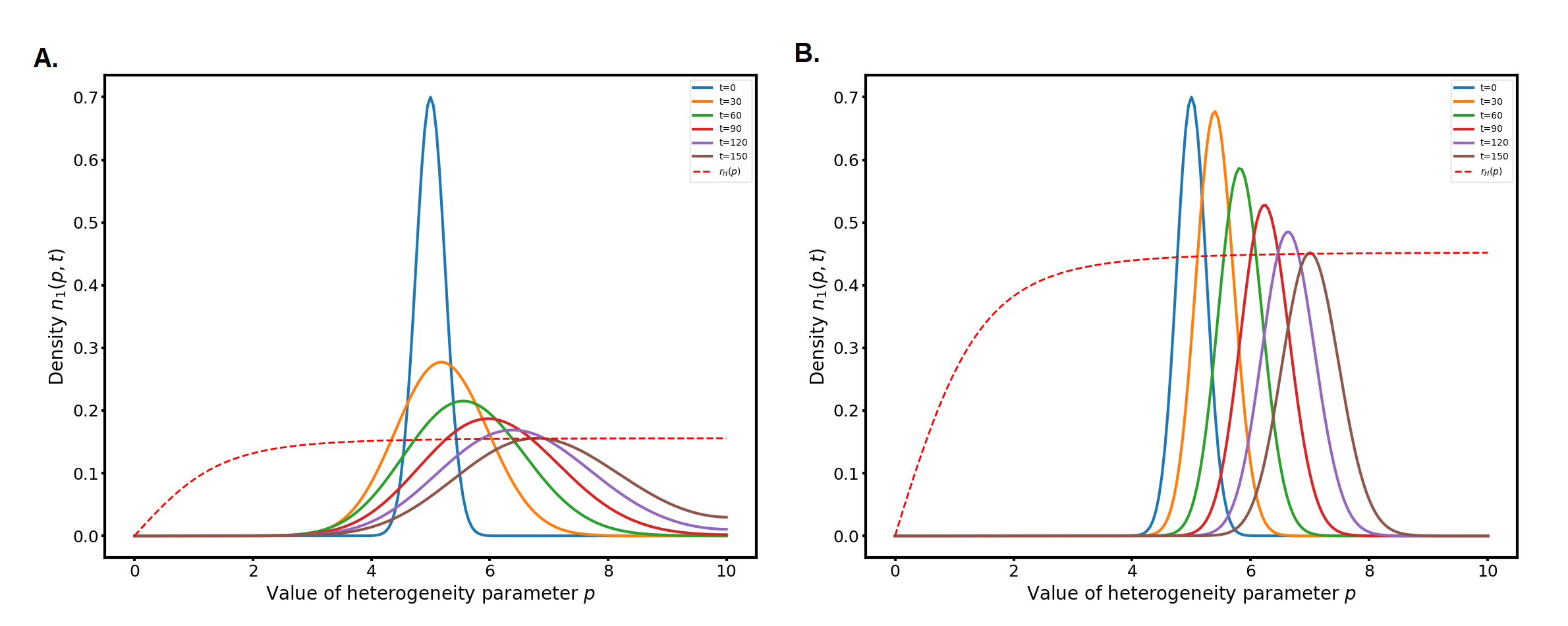}
	\end{center}
	\caption{ Evolution in time of $n_1(p,t)$ from the same initial data for $x=2$ and different values of $\varepsilon$. The dashed line is the Hamiltonian fitness $r_H(p)$ which is re-scaled for the sake of readability. (A.) $\varepsilon = 0.01$ (B.) $\varepsilon = 0.001$. \label{fig:evolution_p_stable}}
\end{figure}

If we fix the value $x=2$ for the timing parameter and take the heterogeneity parameter $p$ as the variable, we can observe (see Figure \ref{fig:evolution_p_stable}) that, according to our theoretical results, the solutions concentrate in the limit $\varepsilon\to0$ on a Dirac measure moving towards infinity. It is due to the function $p\mapsto r_H(p)$ being increasing for mild values of $x$. This implies that in a stable environment, the cells select an optimal adaptation timing $x^*$ for the adaptation to DNA damage and then minimise the variance around it, which amounts to maximising $p$.

\subsubsection{Time-varying environment}

Yet, the experiments on budding yeast cells show that there is a huge variance around the mean adaptation timing. Following \cite{RSX}, we try to explain this discrepancy between the model and reality by adding a varying environment. To make the problem numerically tractable we use
\begin{equation*}
\left\lbrace
\begin{aligned}
    &r_1(p,t) = 1 - D + \bar \alpha(t)D\int_0^\infty \alpha(s) e^{- \gamma_d s - \int_0^s \alpha(z)dz - \int_0^s \beta(p,z)dz}ds ,\\ 
    &\delta_2(p) = D \int_0^\infty \beta(s) e^{- \gamma_d s - \int_0^s \alpha(z)dz - \int_0^s \beta(p,z)dz} d s 
\end{aligned}
\right.
\end{equation*}  
rather than
\begin{equation*}\left\lbrace
\begin{aligned}
&r_1(p,t) = 1 - D + D\int_0^\infty \alpha(s,t) e^{- \gamma_d s - \int_0^s \alpha(z,t)dz - \int_0^s \beta(p,z)dz}ds ,\\
&\delta_2(p,t) = D \int_0^\infty \beta(s) e^{- \gamma_d s - \int_0^s \alpha(z,t)dz - \int_0^s \beta(p,z)dz} d s, 
\end{aligned}\right.
\end{equation*}
because the later requires the program to compute a full vector of integrals at each time step, which makes long time simulations intractable.

\begin{figure}[h!]
	\begin{center}
		\includegraphics[width=350pt]{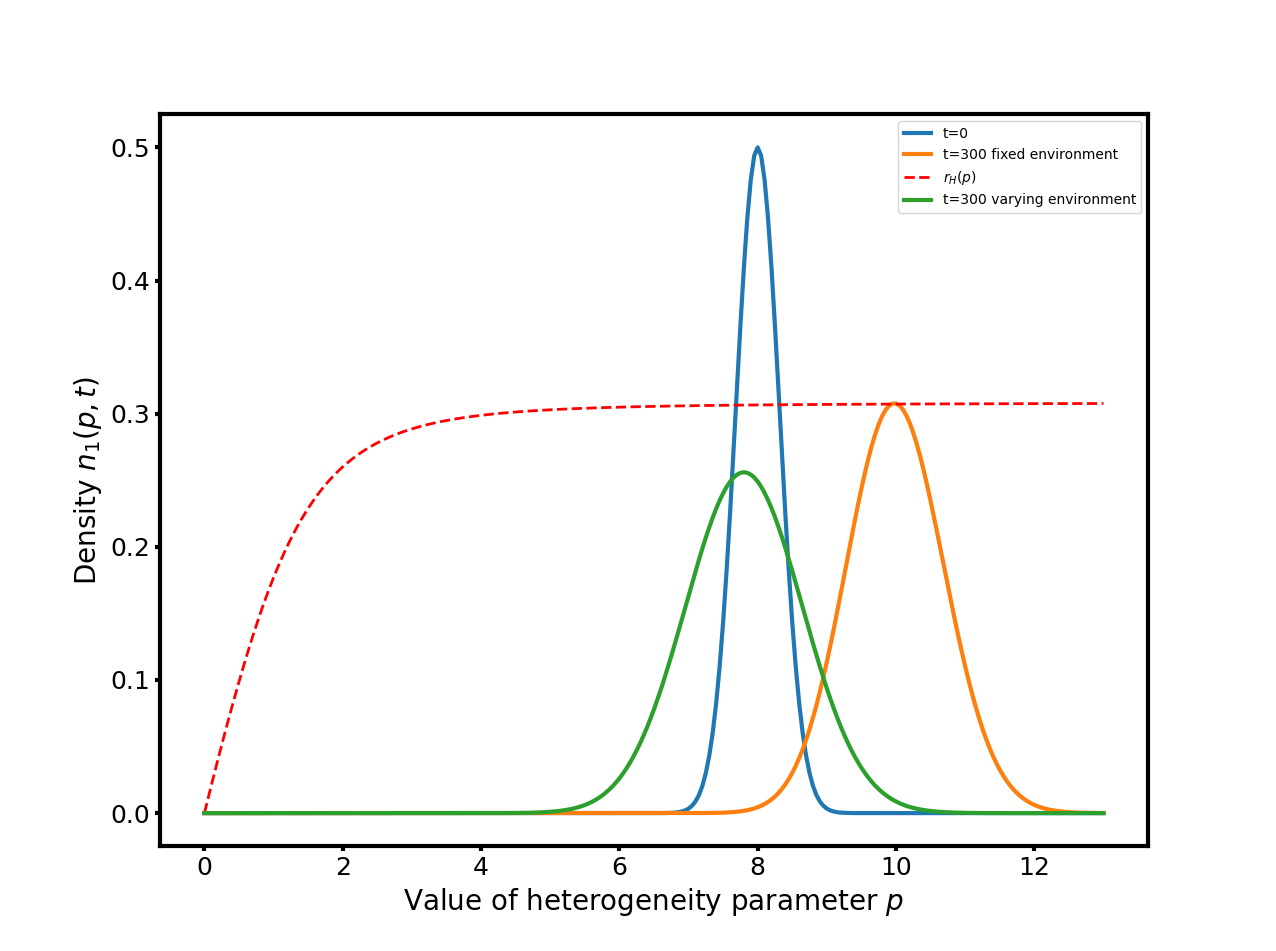}
	\end{center}
	\caption{ Evolution in time of $n_1(p,t)$ for $x=2$, $d_1=d_2=1$ and $\varepsilon=0.001$ in a stable or in a periodically varying environment from the same initial datum. \label{fig:evolution_p_varying}}	
\end{figure}

We choose the time-varying environmental function
\[ \bar \alpha(t) = \cos\left( \frac{\pi t}5  \right)^8\]
and we run the model in both a fixed and a time-varying environment from the same initial datum (see Figure \ref{fig:evolution_p_varying}). We can observe that at $t=300$ the results are very different. In the case of the stable environment, as in Figure \ref{fig:evolution_p_stable}, the mass moves towards $+\infty$. However, with the time-varying environment, the solution moves slowly towards the left. This numerical result strengthens the hypothesis of \cite{RSX} that the heterogeneity in time of the adaptation to DNA process could be due to a bet-hedging mechanism when cells face an unpredictable environment.

\section{Conclusion and perspectives}

In this article, we have investigated a cooperative two-population system of non-local parabolic PDEs motivated by a particular application in genetics: the understanding of the so-called \textit{adaptation to DNA damage} phenomenon. We used a Hamilton-Jacobi approach which is well understood for one population non-local models \cite{BP,BMP,PB}.

First, in order to prove a similar result in our setting, we combined the approach for the one-population model with tools developed in \cite{BES} for non-local systems. We wrote the Hamiltonian associated with the system in terms of an eigenvalue of the matrix $\textbf{D}+\textbf{R}$. After performing a Hopf-Cole transform, we first prove uniform regularity results on the solutions $u_i^\varepsilon$, which allows us to pass to the limit and obtain the constrained Hamilton-Jacobi equation for the limit $u$. When the diffusion coefficients of the two populations are identical, we obtain the additional result that $n_1^\varepsilon/n_2^\varepsilon$ converges in time towards a corrective term $q(x)$ dependent only on $r_1,r_2,\delta_1,\delta_2$.

Then, we have derived from the ODE model of \cite{RSX} a PDE system modelling the evolutionary dynamics of adaptation to DNA damage in a population of eukaryotic cells. Our theoretical results and numerical simulations allow us to support the findings of \cite{RSX} that:
\begin{itemize}
    \item natural selection could be responsible for the apparition of a precise mean timing for the adaptation phenomenon.
    \item the experimentally observed heterogeneity of individual adaptation timings in a population of cells could be explained by a bet-hedging mechanism while facing an unpredictable environment.
\end{itemize}

This study leaves open several questions on both the mathematical and biological sides.

First, our method relies heavily upon the cross terms $\delta_1$ and $\delta_2$ being positive. In the case of a competitive system  or a prey-predator setting, we cannot apply the same techniques. In particular, it is unlikely that the convergence towards a fixed corrective term $q(x)$ will hold true.

We did not address rigorously the question of the long time behaviour of the system. Our heuristic reasoning and the numerical simulations indicate strongly that there is convergence in time of $n_1^\varepsilon+n_2^\varepsilon$ towards a stationary state of the one-population model endowed with the Hamiltonian fitness of the system. A careful analysis is needed to validate this result.

Regarding the bet-hedging explanation for the heterogeneity of adaptation to DNA damage, our theoretical framework has to be adapted to prove solid results. It would be very useful to have a theory able to encompass the same kind of system but with time-varying coefficients as in \cite{FM} for a single species. The question of periodically changing environment is important in many biological applications. It would be especially important to study theoretically and numerically the influence of the time period of those coefficients.

 It would also be interesting to study this kind of two population system in higher dimension. In particular, in our biological setting, it would be interesting to have a bi-dimensional space $(x,p)$ for the genetic trait with Neumann boundary conditions on the boundaries $x=0$ and $p=0$ in order to validate the idea that in a stable environment the solution concentrates on a Dirac mass moving at the same time towards the line $(x^*,p)$ and the direction $p=+\infty$ in the $\epsilon \to 0$ limit. In this bi-dimensional space for the genetic trait, it would also be possible to study the effect of a time-periodic environment in a more realistic framework.
 
 Last, it could be useful to study the more complex model \eqref{eqn:horrible_1}--\eqref{eqn:horrible_4} theoretically and numerically in order to verify that the quasi-static approximation does not hide key features of the biological phenomenon. This might require cumbersome computations and significant computing power, but it remains feasible in principle.

\appendix

\section{Existence and bounds of $N_\varepsilon$}

We prove in this section the existence of a solution of \eqref{HJ:eq:main}. We use the classical Picard Banach fix point Theorem. The details follows the Appendix A of \cite{BMP}.

Let $T>0$ be a given time and $\mathcal{A}$ be the following closed subset:
\[ \mathcal{A} := \left\lbrace \begin{pmatrix} n_1 \\ n_2 \end{pmatrix} \in C([0,T], (L^1_x(\mathbb{R}^+))^2) :\quad  n_i \geq 0, \quad  \int_{\mathbb{R}^+} n_1(x,t) + n_2(x,t) dx < a \right\rbrace\]
where $a = N(0)e^{\frac{C_R T}{\varepsilon}}$ and $C_R =C_r + C_\delta$. Next, we define $\Phi$ the following application 
\[  \begin{aligned}
\Phi \quad  : \ & \ \mathcal{A} &&\rightarrow \quad \mathcal{A}\\
&\begin{pmatrix} n_1 \\ n_2 \end{pmatrix} &&\mapsto \quad  \begin{pmatrix} m_1 \\ m_2 \end{pmatrix}
\end{aligned}\]
where $\textbf{m } = \begin{pmatrix} m_1 \\ m_2 \end{pmatrix}$ is the solution of 
\begin{equation}\label{HJ:eq:App1}
\left\lbrace
\begin{aligned}
& \partial_t \textbf{m} - \varepsilon \textbf{D} \partial_{xx} \textbf{m} =  \frac{1}{\varepsilon}\widetilde{\textbf{R}}(x, N_n) \textbf{m} \\
&\partial_x \textbf{m}(x=0) = 0\\
&\textbf{m}(t= 0) = \textbf{n}_0
\end{aligned}
\right.
\end{equation}
where 
\[ N_n = \int_0^{+\infty} (1 \ 1) \textbf{n}(x,t)dx = \int_{0}^{+\infty}[ n_1(x,t) + n_2(x,t) ] dx \]
and
\[ \widetilde{\textbf{R}}(x,N) =  \left\lbrace 
\begin{aligned}
&\textbf{R}(x, \frac{c_N}{2}) && \text{ if } N<c_N, \\
&\textbf{R}(x,N) &&\text{ if } N \in (\frac{c_N}{2}, C_N), \\
&\textbf{R}(x, 2C_N) && \text{ if } N>C_N. 
\end{aligned}\right. \]
The aim is to verify that $\Phi$ satisfies two claims:
\begin{enumerate}
\item $\Phi$ maps $\mathcal{A}$ into itself, 
\item $\Phi$ is a contractive application for $T$ small enough. 
\end{enumerate}

\vspace{0.4cm}

\textbf{Proof of claim 1. } Let $\textbf{n} \in \mathcal{A}$ and $\textbf{m} = \Phi(\textbf{n})$. By the maximum principle, we have that $m_i \geq 0$. It remains to prove the $L^1$ bound. According to \eqref{HJ:eq:App1}, we have 
\[ \begin{aligned}
\partial_t (\int_0^{+\infty} m_1 (x,t) + m_2(x,t)dx) &= \frac{1}{\varepsilon } \int_0^{+\infty}((r_1(x, N_n) + \delta_2(x))  m_1(x,t)+ (r_2(x, N_n) + \delta_1(x)) m_2(x,t) ) dx \\
&\leq  \frac{C_R}{\varepsilon} \int_{0}^{+\infty} m_1(x,t) + m_2(x,t)dx. 
\end{aligned}\]
Next, we conclude thanks to the Gronwall Lemma that 
\[ \int_{0}^{+\infty} m_1 (x,t) + m_2(x,t)dx \leq N_0 e^{\frac{C_R t}{\varepsilon}} \leq N_0 e^{\frac{C_R T}{\varepsilon}} = a. \]
It finishes the proof of claim 1.

\vspace{0.4cm}

\textbf{Proof of claim 2. } Let $\textbf{n}_1, \textbf{n}_2 \in \mathcal{A}$, $\textbf{m}_1 = \Phi(\textbf{n}_1)$ and $\textbf{m}_2 = \Phi(\textbf{n}_2)$. We have 
\[ \partial_t (\textbf{m}_1 - \textbf{m}_2 ) = \varepsilon \textbf{D} \partial_{xx} (\textbf{m}_1 - \textbf{m}_2 )+ \frac{1}{\varepsilon } \widetilde{\textbf{R}}(x, N_1) (\textbf{m}_1 - \textbf{m}_2) +[\widetilde{\textbf{R}}(x, N_2) - \widetilde{\textbf{R}}(x, N_1) ] \textbf{m}_2 .\]
We recall that $\| \textbf{m}_2 \|_{L^1_x} \leq a$ and $[\widetilde{\textbf{R}}(x, N_2) - \widetilde{\textbf{R}}(x, N_1) ] =(N_2 - N_1) \textbf{I}_2 $. Next, by multiplying at left by $\begin{pmatrix} 1 & 1 \end{pmatrix}$ and integrating over space, we deduce that
\[ \partial_t \| \textbf{m}_1 - \textbf{m}_2 \|_{L^1_x} \leq \frac{C_R}{\varepsilon} \| \textbf{m}_1 - \textbf{m}_2 \|_{L^1_x} + a \| \textbf{n}_1 - \textbf{n}_2 \|_{L^1_x}.\]
Since $\| (\textbf{m}_1 - \textbf{m}_2)(t=0) \|_{L^1_x} = 0$, we conclude thanks to the Gronwall Lemma that 
\[ \| \textbf{m}_1 - \textbf{m}_2 \|_{L^\infty_t, \ L^1_x} \leq \frac{\varepsilon N(0)e^{\frac{C_R T}{\varepsilon}}  }{C_R} (e^{\frac{C_R T}{\varepsilon}} - 1) \| \textbf{n}_1 - \textbf{n}_2 \|_{L^\infty_t, \ L^1_x}. \]
Remarking that $ \frac{\varepsilon N(0)e^{\frac{C_R T}{\varepsilon}}  }{C_R} (e^{\frac{C_R T}{\varepsilon}} - 1)  \to 0$ as $T \to 0$, we conclude that for $T>0$ small enough, $\Phi$ is contractive. It concludes the proof of existence. 

\vspace{0.4cm}

Next, we focus on the bounds of $N_\varepsilon$. According to \eqref{HJ:eq:main}, we have 
\[ \partial_t N_\varepsilon (x,t) =\frac{1}{\varepsilon} \int_{0}^{+\infty} \begin{pmatrix} 1 & 1 \end{pmatrix} \textbf{R}(x, N_\varepsilon) \textbf{n}^\varepsilon (x,t)dx.\]
A direct computation provides that 
\[ \underset{ i, j \in \left\lbrace 1,2 \right\rbrace, \ i \neq j}{\min} ( r_i(x) + \delta_j(x)  -N_\varepsilon )\frac{N_\varepsilon}{\varepsilon} \leq \partial_t N_\varepsilon \leq \underset{ i, j \in \left\lbrace 1,2 \right\rbrace, \ i \neq j}{\max} ( r_i(x) +\delta_j(x)  -N_\varepsilon )\frac{N_\varepsilon}{\varepsilon} .\]
Therefore, since each components of the $\min$ (resp. $\max$) in \eqref{HJ:Hyp:R} is decreasing with respect to $N$, we deduce that if $N_\varepsilon(t_0) = c_N$ then $\partial_t N_\varepsilon(t_0) > 0$ (resp. if $N_\varepsilon (t_0)= C_N$ then $\partial_t N_\varepsilon(t_0) < 0$). The conclusion follows.

\section{Checking the technical hypotheses in the main application}

\begin{lemma}\label{lm:C1}
    Denote
    \begin{equation}\label{h1_b}A_\infty := \int_{0}^{+\infty} \alpha(s)ds < +\infty \end{equation}
    Then, there exists $\kappa\in\R_+^*$ such that
    \[ \forall x\in\R_+, \  \delta_2(x) \geqslant D\e^{-A_\infty}(1+e^{-\kappa x})\e^{-2\gamma_d x}.  \]
\end{lemma}

\begin{proof}
    Note first that the form of $\beta$ implies that
    \begin{equation}\label{h2_b}
    \exists \kappa\in\R_+^*, \ \forall x\in\R_+^*, \ \int_x^{2x} \beta(x,s) ds \geqslant \kappa x. 
    \end{equation}
    Then, we have for all $x\in\R_+$,
    \[
    \delta_2(x) \geqslant D\e^{-A_\infty}\int_{0}^{+\infty}\beta(x,s)\e^{ - \int_0^s \beta(x,z)dz}\e^{-\gamma_d s} ds.
    \]
    By integration by parts we have
    \begin{multline*}
    \int_{0}^{+\infty}\beta(x,s)\e^{ - \int_0^s \beta(x,z)dz}\e^{-\gamma_d s} ds = 1 - \gamma_d \int_{0}^{+\infty} \e^{ - \int_0^s \beta(x,z)dz}\e^{-\gamma_d s} ds\\
    = 1 - \gamma_d\Big(\underbrace{\int_{0}^{2x} \e^{ - \int_0^s \beta(x,z)dz}\e^{-\gamma_d s} ds}_{:=\omega_1(x)} +  \underbrace{\int_{2x}^{+\infty}\e^{ - \int_0^s \beta(x,u)du}\e^{-\gamma_d s} ds}_{:=\omega_2(x)}\Big). 
    \end{multline*}
    Moreover,
    \[\omega_1(x)\leqslant \int_{2x}^{+\infty}\e^{-\gamma_d s} ds = \dfrac{1}{\gamma_d}(1-\e^{-2\gamma_d x}),\]
    and with \eqref{h2_b},
    \[ \omega_2(x) \leqslant \int_{2x}^{+\infty}\e^{ - \int_x^{2x} \beta(x,z)dz}\e^{-\gamma_d s} ds \leqslant \e^{ - \kappa x}\int_{2x}^{+\infty}\e^{-\gamma_d s} ds = \dfrac{\e^{-\kappa x}}{\gamma_d} \e^{-2\gamma_d x} .\]
    Thus, we have
    \[ \int_{0}^{+\infty}\beta(x,s)\e^{ - \int_0^s \beta(x,z)dz}\e^{-\gamma_d s} ds \geqslant (1-\e^{-\kappa x})\e^{-2\gamma_d x}, \] 
    and the result follows.
       \flushright\qed
\end{proof}
~\\

\begin{lemma}\label{lm:C2}
    Assume $D \leqslant 1$, then
    \[ \forall x\in\R_+, \ \delta_2(x) < 1. \]
    and there exist $K_3, K_4\in\R_+^*$ such that
    \[ \forall x\in\R_+, \  \delta_2(x) \leqslant k_4\e^{-K_3 x}.  \]
\end{lemma}

\begin{proof}
    Note that there exists two constants $K_1,K_2\in\R_+^*$ such that
    \[ \forall x\in\R_+^*, \ \int_0^{\frac{x}{2}} \beta(x,s) ds \leqslant K_1 \e^{-K_2 x}, \]
   We have
    \[  \delta_2(x) \leqslant \int_0^{+\infty} \beta(x,s) e^{-\gamma_d s - \int_0^s \alpha(z)dz - \int_0^s \beta(x,z)dz }ds, \]
    By positivity of $\gamma_d$ and $A(\cdot)$,
    \[ \delta_2(x) < \int_0^{+\infty} \beta(x,s)\e^{\int_0^s \beta(x,z)dz }ds = \left[  -\e^{-\int_0^s \beta(x,z)dz}  \right]_{s=0}^{s=+\infty} = 1, \]
    because $\underset{s\to+\infty}{\lim} \int_0^s \beta(x,z)dz = +\infty$.
    
    We compute :
    \[ \delta_2(x) < \underbrace{\int_0^{\frac x 2} \beta(x,s)\e^{-\gamma_d s - \int_0^s \beta(x,z)dz} ds}_{:=\omega_1(x)} + \underbrace{\int_{\frac x 2}^{+\infty} \beta(x,s) \e^{-\gamma_d s - \int_0^s \beta(x,z)dz} ds}_{:=\omega_2(x)}.  \]
    Moreover,
    \[  \omega_1(x) \leqslant \int_0^{\frac x 2}  \beta(x,s) ds \leqslant K_1\e^{-K_2 x},  \]
    and
    \begin{multline*} \omega_2(x) \leqslant \e^{- \frac{ \gamma_d }{2} x} \int_{\frac x 2}^{+\infty}  \beta(x,s) \e^{\int_0^s \beta(x,z)dz }ds\\ = \e^{- \frac{ \gamma_d }{2} x} \left[  -\e^{-\int_0^s \beta(x,z)dz}  \right]_{\frac x 2}^{+\infty} = \e^{- \frac{ \gamma_d }{2} x} \e^{- \int_0^{\frac x 2} \beta(x,s)ds} <  \e^{- \frac{ \gamma_d }{2} x}.    \end{multline*}
    Hence,
    \[ \delta_2\leqslant (1 + K_1) \e^{-\min(\frac{\gamma_d}{2}, K_2) x }.  \]
    
    \flushright\qed
\end{proof}

\vspace{0.5cm}

\noindent{\Large\textbf{Acknowledgements:}} Alexis Leculier and Pierre Roux were supported by the ERC ADORA. Pierre Roux was supported by the Advanced Grant Non local-CPD (Nonlocal PDEs
for Complex Particle Dynamics: Phase Transitions, Patterns and Synchronization) of the European
Research Council Executive Agency (ERC) under the European Union’s Horizon 2020 research and
innovation programme (grant agreement No.883363). Both autors want to warmly thank Beno\^{i}t Perthame for his help and precious advice. Both authors want also to thank Zhou Xu for the time dedicated to fruitful discussions about the biological motivation of this project. 

\footnotesize{\bibliographystyle{abbrv}
\bibliography{Biblio}}

\end{document}